\documentclass[12pt,a4paper]{article}
\usepackage[utf8]{inputenc}
\usepackage[T1]{fontenc}
\usepackage{amsmath}
\usepackage{amsfonts}
\usepackage{amssymb,amsthm}
\usepackage{graphicx}
\usepackage{float}
\usepackage{subcaption}
\usepackage[left=2.00cm, right=2.00cm, top=2.00cm, bottom=3.00cm]{geometry}
\usepackage{cite}

\numberwithin{equation}{section}

\newtheorem{theorem}{Theorem}[section]
\newtheorem{proposition}[theorem]{Proposition}

\theoremstyle{definition}
\newtheorem{example}[theorem]{Example}
\newtheorem{definition}[theorem]{Definition}
\newtheorem{remark}[theorem]{Remark}
\allowdisplaybreaks

\DeclareMathOperator{\arctanh}{arctanh}

\author{}
\title{Singularities of focal sets of pseudo-spherical framed immersions in the three-dimensional anti-de Sitter space}
\author{\\\
        {\bf O. O\u{g}ulcan Tuncer}\thanks{E-mail address:
        otuncer@hacettepe.edu.tr}
 \medskip\\
   \\Department of Mathematics, Hacettepe University,
                    \\ 06800 Beytepe, Ankara, Turkey
}
\begin{document}

\maketitle
\begin{abstract}
We introduce pseudo-spherical non-null framed curves in the three-dimensional anti-de Sitter spacetime and establish the existence and uniqueness of these curves. We then give moving frames along pseudo-spherical framed curves, which are well-defined even at singular points of the curve. These moving frames enable us to define evolutes and focal surfaces of pseudo-spherical framed immersions. We investigate the singularity properties of these evolutes and focal surfaces. We then reveal that the evolute of a pseudo-spherical framed immersion is the set of singular points of its focal surface. We also interpret evolutes and focal surfaces as the discriminant and the secondary discriminant sets of certain height functions, which allows us to explain evolutes and focal surfaces as wavefronts from the viewpoint of Legendrian singularity theory. Examples are provided to flesh out our results, and we use the hyperbolic Hopf map to visualize these examples. \medskip
\end{abstract}
\section{Introduction}
The main tool for investigating the local differential geometry of regular curves is the well-known Frenet frame. This frame is well-defined at any point of the regular curve. However we cannot use the Frenet frame to investigate the local differential geometry of general curves that may have singularities. Therefore having a well-defined moving frame along these curves becomes crucial. The good news is several authors have defined such moving frames for certain families of curves having singularities in different ambient spaces. This journey has begun with the work of Fukunaga and Takahashi \cite{fukunaga}, where they define Legendre curves in the unit tangent bundle of the Euclidean plane and introduce a moving frame called the Legendrian Frenet frame along the Legendre curve. Then the same authors define evolutes and involutes of Legendre curves and Legendre immersions  \cite{fukunaga2,fukunaga3,fukunaga4}. Pedal and contrapedal curves of these Legendre curves have been investigated \cite{li2,tuncer}. These Legendre curves and immersions have also been generalized to different spaces and evolutes, involutes, parallels, pedals, and contrapedals of these curves have been investigated in the Minkowski plane \cite{Sekerci,lisun}, in the Euclidean 2-sphere \cite{li,yu}, in the hyperbolic and de Sitter 2-spaces \cite{chen,li-tuncer}, and in normed planes \cite{vitor}. The local differential geometry of singular curves in higher dimensional spaces has also been studied. Such curves are called the framed curves \cite{honda}. Evolutes and focal surfaces of framed curves in Euclidean 3-space are investigated in \cite{honda2}. Similar problems have also been briefly discussed for the three-sphere \cite{honda3}. See \cite{honda4, li3, lipei, li4, li7, PTY, song,tuncer2,zhang,wang-lang} for other related papers.

%Because it is among the simplest of curved spacetimes, n-dimensional Anti-de- Sitter spacetime (AdS) has been of continuing interest to relativists.

This paper serves two purposes: to extend the regular curves in the anti-de Sitter 3-space to pseudo-spherical framed curves that may have singularities and to investigate evolutes and focal surfaces of these framed curves in terms of singularity theory. The first goal is not hard to achieve; we make use of the Legendrian dualities for pseudo-spheres in semi-Euclidean space with index 2 \cite{CI}. To achieve the second goal, our main tool is a moving frame along the pseudo-spherical framed curve that is well-defined at even singular points. We also use certain height functions to explain evolutes and focal surfaces as wavefronts from the viewpoint of Legendrian singularity theory. 

This paper is organized as follows. In Section 2 we begin with a brief review of the semi-Euclidean 4-space with index 2 and the local differential geometry of regular curves in the anti-de Sitter 3-space. In Section 3 we introduce pseudo-spherical spacelike and timelike framed curves in the anti-de Sitter 3-space. We give the existence and uniqueness theorems for these curves. We also define three types of moving frames along these curves that are isometric under rotations. We then define parallel curves of pseudo-spherical framed curves and show that these parallel curves are also pseudo-spherical framed curves in the anti-de Sitter 3-space. In Section 4 we introduce evolutes and focal surfaces of pseudo-spherical spacelike framed immersions. We show that evolutes are also pseudo-spherical framed immersions. We then obtain the evolute of a pseudo-spherical spacelike framed curve as the set of singular values of the focal surface of the same curve. We also define certain height functions for pseudo-spherical spacelike framed immersions and deduce that the discriminant and secondary discriminant sets of these height functions coincide with the evolute and the focal surface of this immersion. We finally give an example of pseudo-spherical spacelike framed immersion and visualize the projections on the hyperbolic 2-space of this immersion and its evolute by using the hyperbolic Hopf map. In Section 5 we obtain similar results to those in Section 4 for pseudo-spherical timelike framed immersions in the anti-de Sitter 3-space.

\section{Preliminaries}
The semi-Euclidean 4-space with index $2$ denoted by $\mathbb{R}^4_2$ is the real vector space with a pseudo-scalar product given by 
\begin{equation*}
	\langle u,w\rangle=-u_1w_1-u_2w_2+u_3w_3+u_4w_4, 
\end{equation*}
where $u=(u_1,u_2,u_3,u_4),\, w=(w_1,w_2,w_3,w_4)\in\mathbb{R}^4$. \\ 
\indent Vectors in $\mathbb{R}^4_2$ are classified depending on this pseudo-scalar product. Consider a non-zero vector $u=(u_1,u_2,u_3,u_4)\in\mathbb{R}^4_2$. The vector $u$ is called spacelike, timelike, or lightlike (null) if $\langle u,u\rangle>0$, $\langle u,u\rangle<0$ or $\langle u,u\rangle=0$, respectively. The pseudo-norm of the vector $u$ is given by $\|u\|=\sqrt{\lvert\langle u,u\rangle\rvert}$. For three arbitrary vectors $u=(u_1,u_2,u_3,u_4)$, $v=(v_1,v_2,v_3,v_4)$, and $w=(w_1,w_2,w_3,w_4)$, the triple vector product is defined by
\begin{equation*}
	u\times v\times w = 
	\begin{vmatrix}
		-e_1& -e_2 & e_3 & e_4\\
		u_1& u_2 & u_3 &u_4\\
        v_1 & v_2 & v_3 & v_4\\
		w_1& w_2 & w_3 & w_4
	\end{vmatrix} 
\end{equation*}
where the set $\{e_1,e_2,e_3,e_4\}$ is the canonical basis of $\mathbb{R}^4_2$.\\ \indent
In the semi-Euclidean 4-space with index 2, curves are classified depending on their tangent vectors. A curve is said to be spacelike, timelike, or lightlike (null) if the tangent vector of the curve is spacelike, timelike, or lightlike (null), respectively. 

There are three types of pseudo-spheres in the semi-Euclidean 4-space with index 2. The anti-de Sitter 3-space, pseudo $3$-sphere with index 2, and nullcone at the origin
are respectively defined by
\begin{align*}
	&AdS^3=\{\mathbf{u}\in\mathbb{R}^4_2\,\vert\,\langle \mathbf{u},\mathbf{u}\rangle=-1 \},\\
	&S^3_2=\{\mathbf{u}\in\mathbb{R}^4_2\,\vert\,\langle \mathbf{u},\mathbf{u}\rangle=1 \},\\
	&\Lambda^3=\{\mathbf{u}\in\mathbb{R}^4_2\backslash\{\mathbf{0}\} \,\vert\,\langle \mathbf{u},\mathbf{u}\rangle=0 \}.
\end{align*} 
We will make use of the hyperbolic Hopf map $\mathbf{h}$ defined by \cite{Benyounes}
\begin{align}
    \mathbf{h}:&AdS^3 \to H^2(1/2) \nonumber \\
    &(u_1,u_2,u_3,u_4)\mapsto \left(u_1u_3+u_2u_4,u_1u_4-u_2u_3,\dfrac{u_1^2+u_2^2+u_3^2+u_4^2}{2}\right), \label{hopfmap}
\end{align}
where $H^2(1/2)=\{(y_1,y_2,y_3)\in\mathbb{R}^3_1\:|\: y_1^2+y_2^2-y_3^2=-1/4\,\text{and}\,y_3>0 \}$ is the hyperbolic 2-space defined as a surface of constant curvature $-1/4$ in Minkowski space $\mathbb{R}^3_1$.
This map will allow us to get projections on $H^2(1/2)$ of curves in $AdS^3$ and to visualize them.  

We now discuss the local differential geometry of regular curves in the anti-de Sitter 3-space. Let us begin with spacelike curves in $AdS^3$. Let $\gamma:I\to AdS^3$  be a regular spacelike curve. $\gamma'=d\gamma/dt$ is a spacelike vector and $\|\gamma'(t)\|\neq 0$ for all $t\in I$. Since $\gamma$ is a spacelike regular curve, it admits an arc-length parametrization $s=s(t)$. So we may assume that $\gamma(s)$ is a unit-speed curve. Let $T(s)=\gamma'(s)$ be the unit tangent vector. Since $\langle \gamma(s),\gamma(s) \rangle=-1$, we have $\langle \gamma(s), T(s)\rangle=0$. From a direct calculation, we find that $\langle \gamma(s), T'(s)\rangle =-1$. We now take $N_1(s)=T'(s)-\gamma(s)$ and $N_2(s)=\gamma(s)\times T(s)\times N_1(s)$. It is easy to check that $N_1$ and $N_2$ are normal vectors of the spacelike curve $\gamma$ in $AdS^3$. These normal vectors can be spacelike or timelike vectors. We also define the curvature by $\kappa_g(s)=\|T'(s)-\gamma(s)\|$. So we say that the spacelike curve $\gamma$ is a geodesic in $AdS^3$ if $\kappa_g(s)=0$ and $N_1(s)=0$. In the case of $\kappa_g(s)\neq0$, we are able to define the following unit vectors.
\[ n_1(s)=\dfrac{T'(s)-\gamma(s)}{\|T'(s)-\gamma(s)\|}=\dfrac{N_1(s)}{\|N_1(s)\|},\qquad n_2(s)=\gamma(s)\times T(s)\times n_1(s). \]
Hence the set $\{ \gamma(s), T(s), n_1(s), n_2(s)\}$ forms a pseudo-orthonormal frame along the spacelike curve $\gamma$. Then the Frenet-Serret type formulas are governed by
\begin{equation*}
	\begin{pmatrix}
		\gamma'(s)\\
		T'(s)\\
		n_1'(s)\\
        n_2'(s)
	\end{pmatrix}= \begin{pmatrix}
		0 & 1 & 0 & 0 
        \\ 1 & 0 & \kappa_g(s) & 0
        \\ 0 & -\delta \kappa_g(s) & 0 & \tau_g(s) 
        \\ 0 & 0 &  \tau_g(s) & 0
        
	\end{pmatrix}\begin{pmatrix}
		\gamma(s)\\
		T(s)\\
		n_1(s)\\
        n_2(s)
	\end{pmatrix},
\end{equation*}
where $\delta=\langle n_1(s), n_1(s)\rangle$ and $\tau_g(s)=\frac{\delta}{\kappa_g^2(s)}\det(\gamma(s),\gamma'(s),\gamma''(s),\gamma'''(s))$.

We now describe the local differential geometry of regular timelike curves in $AdS^3$. Let $\gamma(s):I\to AdS^3$ be a regular unit-speed timelike curve. Then $T(s)=\gamma'(s)$ is the unit tangent vector to $\gamma$. It is easy to see that $\langle \gamma(s), T(s)\rangle=0$ and then $\langle \gamma(s), T(s)\rangle=1$. The vector $N_1(s)=T'(s)+\gamma(s)$ is pseudo-orthogonal to both $\gamma(s)$ and $T(s)$. Let $N_2(s)=\gamma(s)\times T(s)\times N_1(s)$. It is easy to check that $N_1$ and $N_2$ are spacelike normal vectors of the timelike curve $\gamma$ in $AdS^3$. Define $\kappa_g(s)=\|T'(s)+\gamma(s)\|$. Then in the case of $\kappa_g(s)\neq 0$, we define the unit spacelike vectors
\[ n_1(s)=\dfrac{T'(s)+\gamma(s)}{\|T'(s)+\gamma(s)\|}=\dfrac{N_1(s)}{\|N_1(s)\|},\qquad n_2(s)=\gamma(s)\times T(s)\times n_1(s). \]
The set $\{ \gamma(s), T(s), n_1(s), n_2(s)\}$ forms a pseudo-orthonormal frame along the timelike curve $\gamma$. Then the Frenet-Serret type formulas are given by
\begin{equation*}
	\begin{pmatrix}
		\gamma'(s)\\
		T'(s)\\
		n_1'(s)\\
        n_2'(s)
	\end{pmatrix}= \begin{pmatrix}
		0 & 1 & 0 & 0 
        \\ -1 & 0 & \kappa_g(s) & 0
        \\ 0 &  \kappa_g(s) & 0 & \tau_g(s) 
        \\ 0 & 0 &  -\tau_g(s) & 0
        
	\end{pmatrix}\begin{pmatrix}
		\gamma(s)\\
		T(s)\\
		n_1(s)\\
        n_2(s)
	\end{pmatrix},
\end{equation*}
where $\tau_g(s)=-\frac{1}{\kappa_g^2(s)}\det(\gamma(s),\gamma'(s),\gamma''(s),\gamma'''(s))$.

\section{Pseudo-spherical framed curves in the anti-de Sitter 3-space}
In this section, we consider the local differential geometry of smooth curves in the three-dimensional anti-de Sitter space. If the curve has singular points, we cannot define the pseudo-orthonormal Frenet-type frame at these
singular points given in the previous section. We also cannot use the Frenet–Serret type formulas to study the properties of the original
curve. In order to overcome this obstacle, we shall take advantage of the way developed by \cite{fukunaga, honda2}. So we shall introduce pseudo-spherical framed curves in the three-dimensional anti-de Sitter space. Similar to the regular case, we have two types of these framed curves.
\subsection{Pseudo-spherical spacelike framed curves in $AdS^3$}
Let $\gamma_s:I\to AdS^3$ be a smooth curve. Then $(\gamma_s, v_1,v_2):I\to AdS^3\times \Delta_1$ is called a \textit{pseudo-spherical spacelike framed curve} if $(\gamma_s(s),v_1(s))^*\theta=0$ and $(\gamma_s(s),v_2(s))^*\theta=0$ for all $s\in I$, where
\begin{equation*}
	\Delta_1=\{(\mathbf{u},\mathbf{w})\,\vert\,\langle\mathbf{u},\mathbf{w}\rangle=0\}\subset AdS^3 \times {S}^3_2\,\,(\text{or}\,\,{S}^3_2\times AdS^3)
\end{equation*}
is a $4$-dimensional contact manifold, and $\theta$ is a canonical contact $1$-form on $\Delta_1$ \cite{CI}. The condition $(\gamma_s(s), v_i(s))^*\theta=0$ ($i=1,2$) is equivalent to $\langle \gamma_s'(s),v_i(s)\rangle=0$ ($i=1,2$) for all $s\in I$. If $(\gamma_s, v_1,v_2)$ is an immersion, then it is called a \textit{pseudo-spherical spacelike framed immersion}.

We call $\gamma_s:I\to AdS^3$ a \textit{pseudo-spherical spacelike framed base curve} if there exists a smooth map $(v_1,v_2):I\to\Delta_1$ for which $(\gamma_s, v_1,v_2)$ is a pseudo-spherical spacelike framed curve.

Define $\mu(s)=\gamma_s(s)\times v_1(s)\times v_2(s)$. Then the set $\{\gamma_s(s), v_1(s),v_2(s),\mu(s) \}$ is a pseudo-orthonormal frame along $\gamma_s$. This frame is well-defined even at singular points of $\gamma_s$. The Frenet-Serret type formulas for this frame are given by
\begin{equation}\label{SpacelikeSF}
	\begin{pmatrix}
		\gamma_s'(s)\\
		v_1'(s)\\
		v_2'(s)\\
        \mu'(s)
	\end{pmatrix}= \begin{pmatrix}
		0 & 0 & 0 & \alpha(s) 
        \\ 0 & 0 & \ell(s) & m(s)
        \\ 0 & \ell(s) & 0 & n(s)
        \\ \alpha(s) & -\epsilon m(s) &  \epsilon n(s) & 0
        
	\end{pmatrix}\begin{pmatrix}
	\gamma_s(s)\\
		v_1(s)\\
		v_2(s)\\
        \mu(s)
	\end{pmatrix},
\end{equation}
where $\epsilon=\langle v_1(s),v_1(s)\rangle$, $\alpha(s)=\langle \gamma_s'(s), \mu(s)\rangle$, $\ell(s)=-\epsilon\langle v_1'(s), v_2(s)\rangle$, $m(s)=\langle v_1'(s), \mu(s)\rangle$, and $n(s)=\langle v_2'(s), \mu(s)\rangle$. We call the mapping $(\alpha, \ell, m,n):I\to\mathbb{R}^4$ the \textit{curvature} of the pseudo-spherical spacelike framed curve $(\gamma_s,v_1,v_2)$. Notice that $s_0$ is a singular point of $\gamma_s$ if and only if $\alpha(s_0)=0$.
\begin{definition}
Let $(\gamma, v_1,v_2)$ and $(\Tilde{\gamma}, \Tilde{v}_1, \Tilde{v}_2)$ be two pseudo-spherical spacelike framed curves in $AdS^3$. We say that $(\gamma, v_1,v_2)$ and $(\Tilde{\gamma}, \Tilde{v}_1, \Tilde{v}_2)$ are congruent as pseudo-spherical framed curves if there exists a matrix $A\in SO(2,2)$ such that for all $s$
\[ \Tilde{\gamma}(s)=A(\gamma(s)),\quad \Tilde{v}_1(s)=A(v_1(s)),\quad \Tilde{v}_2(s)=A(v_2(s)). \]
\end{definition}
\begin{theorem}[Existence of pseudo-spherical spacelike framed curves]
For a smooth mapping $(\alpha,\ell,m,n):I\to\mathbb{R}^4$, there exists a pseudo-spherical spacelike framed curve $(\gamma_s,v_1,v_2)$ such that $\alpha$, $\ell$, $m$, and $n$ are the curvatures of $\gamma_s$.
\end{theorem}
\begin{proof}
There are two cases we need to deal with since $n_1$ and $n_2$ can be spacelike or timelike. However, we only prove the existence of pseudo-spherical spacelike framed curve whose $n_1$ is timelike and so $n_2$ is spacelike. The other case follows quite similarly. Choose a fixed value $s=s_0$ of the parameter. We consider the initial value problem 
\begin{equation}\label{exist1}
    \dfrac{d}{ds}F(s)=A(s)F(s),\quad F(s_0)=\eta,
\end{equation}
where $F(s)\in\mathbb{R}^{4\times 4}$ is a matrix, $\eta=\text{diag}[-1,-1,1,1]$, and 
\[ A(s)=\begin{pmatrix}
		0 & 0 & 0 & \alpha(s) 
        \\ 0 & 0 & \ell(s) &  m(s)
        \\ 0 & \ell(s) & 0 & n(s)
        \\ \alpha(s) & m(s) &  -n(s) & 0
	\end{pmatrix}. \]
 By the existence theorem of a solution of a system of linear ordinary differential equations, there exists a solution $F(s)$. It is easy to see that $A(s)\in \mathfrak{o}(2,2)$ that is $\eta A(s)+A^t(s) \eta=0$, where $\cdot^t$ stands for the matrix transpose. Using this equality and \eqref{exist1} yields
 \begin{align*}
     \dfrac{d}{ds}(F^t(s)\eta F(s))&=\left( \dfrac{d}{ds} F(s)\right)\eta F(s)+F^t(s)\eta \left( \dfrac{d}{ds}F(s) \right) \\
     &=(A(s)F(s))^t\eta F(s)+F^t(s)\eta A(s)F(s)\\
     &=F^t(s)(\eta A(s)+A^t(s)\eta)F(s)\\
     &=0.
 \end{align*}
 Therefore $F^t(s)\eta F(s)$ is constant. Then we have $F^t(s)\eta F(s)=F^t(s_0)\eta F(s_0)=\eta$. This means that $F(s)$ is a semi-orthogonal matrix. Now set $F(s)= \begin{pmatrix}
 \gamma(s) & v_1(s) & v_2(s)& \mu(s)
 \end{pmatrix}^t.$ Taking determinant on both sides of $F^t(s)\eta F(s)=\eta$ and then differentiating the resulting equation, we get $d/ds\left( \det(F(s)) \right)=0$. Then we find that 
 \[ \det(F(s))=\det(F(s_0))=\det(\eta)=1. \]
 Therefore $F(s)\in SO(2,2)$ and $\gamma(s)\times v_1(s)\times v_2(s)=\mu(s)$. Next consider the following initial value problem
 \[ \gamma'(s)=\alpha(s)\mu(s),\quad \gamma(s_0)=x\in AdS^3. \]
 By the existence theorem of a solution of a system of linear ordinary differential equations, there exists a solution $\gamma(s)$. Finally we conclude that there exists a pseudo-spherical spacelike framed curve $(\gamma_s,v_1,v_2):I\to AdS^3\times \Delta_1$ whose curvature is $(\alpha,\ell,m,n)$.
\end{proof}
 We have seen that the proof of the existence of these curves is similar to its counterparts; regular space curves \cite{gray}, Legendre curves in the unit tangent bundle \cite{fukunaga}, and framed curves \cite{honda}. However, the proof of the uniqueness part differs from these Euclidean counterparts since the latter make use of the positiveness of the Euclidean metric. A similar proof is given for timelike curves in Minkowski spacetime \cite{ForRom}. 
\begin{theorem}[Uniqueness of pseudo-spherical spacelike framed curves]
Let $(\gamma_s, v_1, v_2)$ and $(\Tilde{\gamma}_s, \Tilde{v}_1, \Tilde{v}_2)$ be two pseudo-spherical spacelike framed curves in $AdS^3$. Suppose that the curvatures $(\alpha,\ell,m,n)$ and $(\Tilde{\alpha}, \Tilde{\ell},\Tilde{m}, \Tilde{n})$ of these two framed curves coincide. Then $(\gamma_s, v_1, v_2)$ and $(\Tilde{\gamma}_s, \Tilde{v}_1, \Tilde{v}_2)$ are congruent as pseudo-spherical spacelike framed curves.
%Let $(\gamma_s, n_1, n_2)$ be a pseudo-spherical spacelike framed curves  in $AdS^3$ whose curvature is $(\alpha, \ell, m,n)$. Any other pseudo-spherical spacelike framed curve $\Tilde{\gamma}$ in $AdS^3$ with the same curvature $(\alpha, \ell, m,n)$ is congruent to $\gamma$. 
\end{theorem}
\begin{proof}
Let $\{\gamma_s(s_0), n_1(s_0), n_2(s_0), \mu(s_0)\}$ and $\{\Tilde{\gamma}_s(s_0), \Tilde{n}_1(s_0), \Tilde{n}_2(s_0), \Tilde{\mu}(s_0)\}$ be the pseudo-orthonormal frames at $s_0\in I$ of $\gamma$ and $\Tilde{\gamma}$, respectively. It is always possible by using a transformation matrix $A\in SO(2,2)$ to set $\gamma(s_0)=\Tilde{\gamma}(s_0)$, $n_1(s_0)=\Tilde{n}_1(s_0)$, and $n_2(s_0)=\Tilde{n}_2(s_0)$. Then by definition we have $\mu(s_0)=\Tilde{\mu}(s_0)$. Since the curvatures of $\gamma$ and $\Tilde{\gamma}$ are coincident, we have 
\begin{equation*}
	\begin{pmatrix}
		\gamma_s'(s)\\
		v_1'(s)\\
		v_2'(s)\\
        \mu'(s)
	\end{pmatrix}= \begin{pmatrix}
		0 & 0 & 0 & \alpha(s) 
        \\ 0 & 0 & \ell(s) & m(s)
        \\ 0 & \ell(s) & 0 & n(s)
        \\ \alpha(s) & -\epsilon m(s) &  \epsilon n(s) & 0
        
	\end{pmatrix}\begin{pmatrix}
	\gamma_s(s)\\
		v_1(s)\\
		v_2(s)\\
        \mu(s)
	\end{pmatrix},
\end{equation*}
\begin{equation*}
	\begin{pmatrix}
		\Tilde{\gamma}_s'(s)\\
		 \Tilde{v}_1'(s)\\
		\Tilde{v}_2'(s)\\
        \Tilde{\mu}'(s)
	\end{pmatrix}= \begin{pmatrix}
		0 & 0 & 0 & \alpha(s) 
        \\ 0 & 0 & \ell(s) & m(s)
        \\ 0 & \ell(s) & 0 & n(s)
        \\ \alpha(s) & -\Tilde{\epsilon} m(s) &  \Tilde{\epsilon} n(s) & 0
	\end{pmatrix}\begin{pmatrix}
	\Tilde{\gamma}_s(s)\\
		 \Tilde{v}_1(s)\\
		\Tilde{v}_2(s)\\
        \Tilde{\mu}(s)
	\end{pmatrix},
\end{equation*}
where $\Tilde{\epsilon}=\epsilon$. These two equations can be written in a more compact form as
\begin{equation} \label{uniq1}
    \dfrac{d}{ds}F=A(s)F(s),
\end{equation}
\begin{equation} \label{uniq2}
    \dfrac{d}{ds}\Tilde{F}=A(s)\Tilde{F}(s).
\end{equation}
Notice that the frames $\{ \gamma,v_1,v_2,\mu\}$ and $\{\Tilde{\gamma}, \Tilde{v}_1, \Tilde{v}_2, \Tilde{\mu}\}$ are related by 
\begin{equation}
	\begin{pmatrix}
		\gamma_s(s)\\
		v_1(s)\\
		v_2(s)\\
        \mu(s)
	\end{pmatrix}= \begin{pmatrix}
		-\langle \gamma, \Tilde{\gamma} \rangle & \epsilon\langle \gamma, \Tilde{v}_1 \rangle & -\epsilon \langle \gamma, \Tilde{v}_2 \rangle & \langle \gamma, \Tilde{\mu} \rangle 
        \\ -\langle v_1, \Tilde{\gamma} \rangle & \epsilon\langle v_1, \Tilde{v}_1 \rangle & -\epsilon \langle v_1, \Tilde{v}_2 \rangle & \langle v_1, \Tilde{\mu} \rangle 
        \\ -\langle v_2, \Tilde{\gamma} \rangle & \epsilon\langle v_2, \Tilde{v}_1 \rangle & -\epsilon \langle v_2, \Tilde{v}_2 \rangle & \langle v_2, \Tilde{\mu} \rangle 
        \\ -\langle \mu, \Tilde{\gamma} \rangle & \epsilon\langle \mu, \Tilde{v}_1 \rangle & -\epsilon \langle \mu, \Tilde{v}_2 \rangle & \langle \mu, \Tilde{\mu} \rangle 
	\end{pmatrix}\begin{pmatrix}
	\Tilde{\gamma}_s(s)\\
		 \Tilde{v}_1(s)\\
		\Tilde{v}_2(s)\\
        \Tilde{\mu}(s)
	\end{pmatrix}.
\end{equation}
So we have
\begin{equation} \label{uniq3}
    F(s)=B(s)\Tilde{F}(s).
\end{equation}
Notice that $B(s_0)=I_4$, where $I_4$ is the $4\times 4$ identity matrix. Differentiating \eqref{uniq3} and introducing \eqref{uniq1} and \eqref{uniq2}, we obtain a system of first-order differential equations given by
\[ \dfrac{d}{ds}B(s)+B(s)A(s)-A(s)B(s)=0. \]
By assumption elements of $A(s)$ are differentiable functions. Therefore, this system of first-order differential equations admits a unique solution. It is easy to check that $B(s)=I_4$ is a solution of this system for all $s\in I$. Hence, this is the only solution. So we find that $U=\Tilde{U}$. 
\end{proof}
\begin{proposition}\label{propsparametrization}
    If $(\gamma_s, v_1,v_2)$ is a pseudo-spherical spacelike framed curve with the curvature $(\alpha, \ell, m,n)$, then $(\alpha, \ell, m,n)$ depends on the parametrization of $(\gamma_s, v_1,v_2)$.
\end{proposition}
\begin{proof}
Let $I$ and $\Tilde{I}$ be intervals. A smooth function $u:\Tilde{I}\to I$ is a (positive) change of parameter if $u$ is surjective and has positive derivatives at every point.
    Let $(\tilde{\gamma}_s, \tilde{v}_1,\tilde{v}_2)$ be a pseudo-spherical spacelike framed curve and let $(\tilde{\alpha}, \tilde{\ell}, \tilde{m},\tilde{n})$ be the curvature of this curve.  Suppose that $(\gamma_s, v_1,v_2)$ and $(\tilde{\gamma}_s, \tilde{v}_1,\tilde{v}_2)$ are parametrically equivalent by the change of parameter
    $u:\Tilde{I}\to I$, that is, $(\tilde{\gamma}_s(s), \tilde{v}_1(s),\tilde{v}_2(s))=(\gamma_s(u(s)), v_1(u(s)),v_2(u(s)))$ for all $s\in \Tilde{I}$. Then using \eqref{SpacelikeSF}
   \begin{equation*}
       (\tilde{\alpha}(s), \tilde{\ell}(s), \tilde{m}(s),\tilde{n}(s))=u'(s)(\alpha(u(s)), \ell(u(s)), m(u(s)),n(u(s))).
   \end{equation*}
   Hence the curvature depends on the parametrization.
\end{proof}
Let $(\gamma_s,v_1,v_2):I\to AdS^3\times \Delta_1$ be a pseudo-spherical spacelike framed curve with the curvature $(\alpha,\ell,m,n)$. Now we shall show that it is possible to construct a new frame along $\gamma_s$ similar to the Bishop frame \cite{Bishop} by leaving $\mu$ fixed and rotating $(v_1,v_2)$. For a smooth function $\theta(s):I\to\mathbb{R}$, define $(\Bar{v}_1, \Bar{v}_2)\in \Delta_1$ by 
\[ \begin{pmatrix}
    \Bar{v}_1(s) \\ 
    \Bar{v}_2(s)
\end{pmatrix}=\begin{pmatrix}
    \cosh\theta(s) & \sinh\theta(s) \\
    \sinh\theta(s) & \cosh\theta(s)
\end{pmatrix}\begin{pmatrix}
    v_1(s) \\ 
    v_2(s)
\end{pmatrix}. \]
Then $(\gamma_s, \Bar{v}_1, \Bar{v}_2):I\to AdS^3\times \Delta_1$ is also a pseudo-spherical spacelike framed curve. It is easy to see that $\Bar{\mu}=\gamma_s(s)\times\Bar{v}_1(s)\times\Bar{v}_2(s)=\mu(s)$. So the set $\{ \gamma_s, \Bar{v}_1, \Bar{v}_2,\mu(s)\}$ is a pseudo-orthonormal frame along $\gamma_s$. Using the formulas in \eqref{SpacelikeSF}, we find that 
\begin{align*}
    \Bar{v}_1(s)&=(\theta'(s)+\ell(s))\sinh\theta(s)\,v_1(s)+(\theta'(s)+\ell(s))\cosh\theta(s)\,v_2(s)\\
    &\qquad+(m(s)\cosh\theta(s)+n(s)\sinh\theta(s))\mu(s),\\
    \Bar{v}_2(s)&=(\theta'(s)+\ell(s))\cosh\theta(s)\,v_1(s)+(\theta'(s)+\ell(s))\sinh\theta(s)\,v_2(s)\\
    &\qquad+(m(s)\sinh\theta(s)+n(s)\cosh\theta(s))\mu(s).
\end{align*}
Now set $\theta'(s)=-\ell(s)$. In this case we call the set $\{ \gamma_s, \Bar{v}_1, \Bar{v}_2,\mu(s)\}$ the \textit{Bishop-type frame} along $\gamma_s$. We have the following derivative formulas.
\begin{equation}\label{SpacelikeSF}
	\begin{pmatrix}
		\gamma_s'(s)\\
	\Bar{v}_1'(s)\\
		\Bar{v}_2'(s)\\
        \mu'(s)
	\end{pmatrix}= \begin{pmatrix}
		0 & 0 & 0 & \alpha(s) 
        \\ 0 & 0 & 0 & \Bar{m}(s)
        \\ 0 & 0 & 0 & \Bar{n}(s)
        \\ \alpha(s) & -\epsilon \Bar{m}(s) &  \epsilon \Bar{n}(s) & 0
        
	\end{pmatrix}\begin{pmatrix}
	\gamma_s(s)\\
		\Bar{v}_1(s)\\
		\Bar{v}_2(s)\\
        \mu(s)
	\end{pmatrix},
\end{equation}
where
\[ \begin{pmatrix}
    \Bar{m}(s) \\ 
    \Bar{n}(s)
\end{pmatrix}=\begin{pmatrix}
    \cosh\theta(s) & \sinh\theta(s) \\
    \sinh\theta(s) & \cosh\theta(s)
\end{pmatrix}\begin{pmatrix}
    m(s) \\ 
    n(s)
\end{pmatrix}. \]
Let $(\gamma_s,v_1,v_2)$ be a pseudo-spherical spacelike framed immersion. 
We will mainly use another moving frame along $\gamma_s$ defined as follows. Let $(f_1,f_2)\in \Delta_1$ defined by 
\[ f_1(s)=\dfrac{n(s)v_1(s)-m(s)v_2(s)}{\sqrt{|n^2(s)-m^2(s)|}},\quad f_2(s)=\dfrac{-m(s)v_1(s)+n(s)v_2(s)}{\sqrt{|n^2(s)-m^2(s)|}} \]
where $m^2(s)\neq n^2(s)$ for all $s\in I$. Then $(\gamma_s,f_1,f_2):I\to AdS^3\times \Delta_1$ is also a pseudo-spherical spacelike framed immersion and $\{\gamma_s,f_1,f_2,\mu=\gamma_s\times f_1\times f_2\}$ is a pseudo-orthonormal frame along $\gamma_s$. The derivative formulas for this new frame are 
\begin{equation}\label{SpacelikeNSF}
	\begin{pmatrix}
		\gamma_s'(s)\\
	f_1'(s)\\
		f_2'(s)\\
        \mu'(s)
	\end{pmatrix}= \begin{pmatrix}
		0 & 0 & 0 & \alpha(s) 
        \\ 0 & 0 & \hat{\ell}(s) & 0
        \\ 0 & \hat{\ell}(s) & 0 & \hat{n}(s)
        \\ \alpha(s) &  0 &   \hat{\epsilon}\hat{n}(s) & 0
        
	\end{pmatrix}\begin{pmatrix}
	\gamma_s(s)\\
		f_1(s)\\
		f_2(s)\\
        \mu(s)
	\end{pmatrix},
\end{equation}
where $\hat{\epsilon}=\langle f_1,f_1\rangle$,
\[ \hat{\ell}(s)=\ell(s)+\dfrac{m(s)n'(s)-n(s)m'(s)}{n^2(s)-m^2(s)},\quad \hat{n}(s)=\hat{\epsilon}\epsilon\sqrt{|n^2(s)-m^2(s)|}.  \]

We close this section by defining anti-de Sitter parallel curves of pseudo-spherical spacelike framed curves. Let $(\gamma_s,v_1,v_2):I\to AdS^3\times \Delta_1$ be a pseudo-spherical spacelike framed curve with the curvature $(\alpha,\ell,m,n)$. We define a mapping $\gamma_s^\phi:I\to AdS^3$ by
\begin{equation}\label{spaceparallel}
    \gamma_s^\phi(s)=f(\phi)\gamma_s(s)+g(\phi)\left( P(\theta(s)) v_1(s)+R(\theta(s))v_2(s) \right),
\end{equation}
where $\phi$ is a fixed real number, $\theta'(s)=-\ell(s)$, and 
\[ \big(P(\theta(s)),R(\theta(s))\big)\in\{ (\cosh\theta(s),\sinh\theta(s)), (\sinh\theta(s),\cosh\theta(s)) \}. \] 
Moreover, if $\epsilon(P^2(\theta(s))-R^2(\theta(s)))=-1$, then $f(\phi)=\cos\phi$ and $g(\phi)=\sin\phi$, and if $\epsilon(P^2(\theta(s))-R^2(\theta(s)))=1$, then $f(\phi)=\cosh\phi$ and $g(\phi)=\sinh\phi$. This mapping is called the anti-de Sitter parallel of $\gamma_s$. 
\begin{proposition}\label{spaceparallelprop}
    Let $(\gamma_s,v_1,v_2):I\to AdS^3\times \Delta_1$ be a pseudo-spherical spacelike framed curve with the curvature $(\alpha,\ell,m,n)$. For a fixed real number $\phi$ and $\theta'(s)=-\ell(s)$, the anti-de Sitter parallel $(\gamma_s^\phi,v_1^\phi,v_2^\phi):I\to AdS^3\times \Delta_1$ is a pseudo-spherical spacelike framed curve with the curvature $(\alpha^\phi,0,m^\phi, n^\phi)$, where
    \begin{align*}
        v_1^\phi(s)&=\epsilon(P^2(\theta(s))-R^2(\theta(s))) g(\phi)\gamma_s(s)+f(\phi)\left(P(\theta(s)) v_1(s)+R(\theta(s))v_2(s)  \right),\\
        v_2^\phi(s)&=R(\theta(s)) v_1(s)+P(\theta(s))v_2(s),\\
        \alpha^\phi(s)&=\alpha(s)f(\phi)+g(\phi)\big(P(\theta(s)) m(s)+R(\theta(s))n(s)\big),\\
        m^\phi(s)&= \epsilon(P^2(\theta(s))-R^2(\theta(s))) g(\phi)\alpha(s)+f(\phi)(P(\theta(s)) m(s)+R(\theta(s))n(s)),\\
        n^\phi(s)&= R(\theta(s)) m(s)+P(\theta(s))n(s).
    \end{align*}
\end{proposition}
\begin{proof}
It is easy to see that $\langle \gamma_s^\phi, v_1^\phi\rangle=0$, $\langle \gamma_s^\phi, v_2^\phi\rangle=0$, and $\langle v_1^\phi, v_2^\phi\rangle=0$. The derivative of \eqref{spaceparallel} gives
\[ (\gamma_s^\phi)'(s)=\left(\alpha(s)f(\phi)+g(\phi)\big(P(\theta(s)) m(s)+R(\theta(s))n(s) \big)\right)\mu(s). \]
Hence $\langle (\gamma_s^\phi)', v_1^\phi\rangle=0$ and $\langle (\gamma_s^\phi)', v_2^\phi\rangle=0$. So $(\gamma_s^\phi,v_1^\phi,v_2^\phi)$ is a pseudo-spherical spacelike framed curve. We also have $\mu^\phi(s)=\mu(s)$. Therefore $\alpha^\phi$ is just the coefficient of $\mu(s)$ in the above equation. An easy calculation shows that
\begin{align*}
    (v_1^\phi)'(s)&=\left(\epsilon(P^2(\theta(s))-R^2(\theta(s))) g(\phi)\alpha(s)+f(\phi)(P(\theta(s)) m(s)+R(\theta(s))n(s))\right)\mu(s),\\
    (v_2^\phi)'(s)&= \left( R(\theta(s)) m(s)+P(\theta(s))n(s) \right)\mu(s),
\end{align*}
which concludes the proof. Note that if $(\gamma_s,v_1,v_2)$ is a pseudo-spherical spacelike framed immersion, then $(\gamma_s^\phi,v_1^\phi,v_2^\phi)$ is also a pseudo-spherical spacelike framed immersion.
\end{proof}
\subsection{Pseudo-spherical timelike framed curves in $AdS^3$}
Let $\gamma_t:I\to AdS^3$ be a smooth curve. Then $(\gamma_t, v_1,v_2):I\to AdS^3\times \Delta_5$ is called a \textit{pseudo-spherical timelike framed curve} if $(\gamma_t(s),v_1(s))^*\phi=0$ and $(\gamma_t(s),v_2(s))^*\phi=0$ for all $s\in I$, where
\begin{equation*}
	\Delta_5=\{(\mathbf{u},\mathbf{w})\,\vert\,\langle\mathbf{u},\mathbf{w}\rangle=0\}\subset S^3_2 \times S^3_2
\end{equation*}
is a $4$-dimensional contact manifold, and $\phi$ is a canonical contact $1$-form on $\Delta_5$ \cite{CI}. The condition $(\gamma_t(s), v_i(s))^*\phi=0$ ($i=1,2$) is equivalent to $\langle \gamma_t'(s),v_i(s)\rangle=0$ ($i=1,2$) for all $s\in I$. If $(\gamma_t, v_1,v_2)$ is an immersion, then it is called a \textit{pseudo-spherical timelike framed immersion}.

We call $\gamma_t:I\to AdS^3$ a \textit{pseudo-spherical timelike framed base curve} if there exists a smooth map $(v_1,v_2):I\to\Delta_5$ for which $(\gamma_t, v_1,v_2)$ is a pseudo-spherical timelike framed curve.

Let $\mu(s)=\gamma_t(s)\times v_1(s)\times v_2(s)$. Then the set $\{\gamma_t(s), v_1(s),v_2(s),\mu(s) \}$ is a pseudo-orthonormal frame along $\gamma_t$. This frame is well-defined even at singular points of $\gamma_t$. The Frenet-Serret type formulas for this frame are 
\begin{equation}\label{TimelikeSF}
	\begin{pmatrix}
		\gamma_t'(s)\\
		v_1'(s)\\
		v_2'(s)\\
        \mu'(s)
	\end{pmatrix}= \begin{pmatrix}
		0 & 0 & 0 & \alpha(s) 
        \\ 0 & 0 & \ell(s) & m(s)
        \\ 0 & -\ell(s) & 0 & n(s)
        \\ -\alpha(s) &  m(s) &  n(s) & 0
        
	\end{pmatrix}\begin{pmatrix}
	\gamma_t(s)\\
		v_1(s)\\
		v_2(s)\\
        \mu(s)
	\end{pmatrix},
\end{equation}
where $\alpha(s)=-\langle \gamma_t'(s), \mu(s)\rangle$, $\ell(s)=\langle v_1'(s), v_2(s)\rangle$, $m(s)=-\langle v_1'(s), \mu(s)\rangle$, and $n(s)=-\langle v_2'(s), \mu(s)\rangle$. We call the mapping $(\alpha, \ell, m,n):I\to\mathbb{R}^4$ the \textit{curvature} of the pseudo-spherical timelike framed curve $(\gamma_t,v_1,v_2)$. Notice that $s_0$ is a singular point of $\gamma_t$ if and only if $\alpha(s_0)=0$.
\begin{definition}
Let $(\gamma, v_1,v_2)$ and $(\Tilde{\gamma}, \Tilde{v}_1, \Tilde{v}_2)$ be two pseudo-spherical timelike framed curves in $AdS^3$. We say that $(\gamma, v_1,v_2)$ and $(\Tilde{\gamma}, \Tilde{v}_1, \Tilde{v}_2)$ are congruent as pseudo-spherical framed curves if there exists a matrix $A\in SO(2,2)$ such that for all $s$
\[ \Tilde{\gamma}(s)=A(\gamma(s)),\quad \Tilde{v}_1(s)=A(v_1(s)),\quad \Tilde{v}_2(s)=A(v_2(s)). \]
\end{definition}
\begin{theorem}[Existence and uniqueness of pseudo-spherical timelike framed curves]
For a smooth mapping $(\alpha,\ell,m,n):I\to\mathbb{R}^4$, there exists a pseudo-spherical timelike framed curve $(\gamma_t,v_1,n_2)$ such that $\alpha$, $\ell$, $m$, and $n$ are the curvatures of $\gamma_t$. Any other pseudo-spherical timelike framed curve $\Tilde{\gamma}_t$ in $AdS^3$ with the same curvature $(\alpha, \ell, m,n)$ is congruent to $\gamma_t$. 
\end{theorem}
\begin{remark}
Note that the curvature of a pseudo-spherical timelike framed curve depends on the parametrization. This can be easily proved similar to Proposition \ref{propsparametrization}. 
\end{remark}
Let $(\gamma_t,v_1,v_2):I\to AdS^3\times \Delta_5$ be a pseudo-spherical timelike framed curve with the curvature $(\alpha,\ell,m,n)$. For a smooth function $\theta(s):I\to\mathbb{R}$, define $(\Bar{v}_1, \Bar{v}_2)\in \Delta_5$ by 
\[ \begin{pmatrix}
    \Bar{v}_1(s) \\ 
    \Bar{v}_2(s)
\end{pmatrix}=\begin{pmatrix}
    \cos\theta(s) & -\sin\theta(s) \\
    \sin\theta(s) & \cos\theta(s)
\end{pmatrix}\begin{pmatrix}
    v_1(s) \\ 
    v_2(s)
\end{pmatrix}. \]
Then $(\gamma_t, \Bar{v}_1, \Bar{v}_2):I\to AdS^3\times \Delta_5$ is also a pseudo-spherical timelike framed curve. It is easy to see that $\Bar{\mu}=\gamma_t(s)\times\Bar{v}_1(s)\times\Bar{v}_2(s)=\mu(s)$. So the set $\{ \gamma_t, \Bar{v}_1, \Bar{v}_2,\mu(s)\}$ is a pseudo-orthonormal frame along $\gamma_t$. Using the formulas in \eqref{TimelikeSF}, we find that 
\begin{align*}
    \Bar{v}_1(s)&=(-\theta'(s)+\ell(s))\sin\theta(s)\,v_1(s)+(-\theta'(s)+\ell(s))\cos\theta(s)\,v_2(s)\\
    &\qquad+(m(s)\cos\theta(s)-n(s)\sin\theta(s))\mu(s),\\
    \Bar{v}_2(s)&=(\theta'(s)-\ell(s))\cos\theta(s)\,v_1(s)+(\theta'(s)-\ell(s))\sin\theta(s)\,v_2(s)\\
    &\qquad+(m(s)\sin\theta(s)+n(s)\cos\theta(s))\mu(s).
\end{align*}
Now set $\theta'(s)=\ell(s)$. In this case we call the set $\{ \gamma_s, \Bar{v}_1, \Bar{v}_2,\mu(s)\}$ the \textit{Bishop-type frame} along $\gamma_t$. We have 
\begin{equation}\label{SpacelikeSF}
	\begin{pmatrix}
		\gamma_t'(s)\\
	\Bar{v}_1'(s)\\
		\Bar{v}_2'(s)\\
        \mu'(s)
	\end{pmatrix}= \begin{pmatrix}
		0 & 0 & 0 & \alpha(s) 
        \\ 0 & 0 & 0 & \Bar{m}(s)
        \\ 0 & 0 & 0 & \Bar{n}(s)
        \\ -\alpha(s) &  \Bar{m}(s) &   \Bar{n}(s) & 0
        
	\end{pmatrix}\begin{pmatrix}
	\gamma_t(s)\\
		\Bar{v}_1(s)\\
		\Bar{v}_2(s)\\
        \mu(s)
	\end{pmatrix},
\end{equation}
where
\[ \begin{pmatrix}
    \Bar{m}(s) \\ 
    \Bar{n}(s)
\end{pmatrix}=\begin{pmatrix}
    \cos\theta(s) & -\sin\theta(s) \\
    \sin\theta(s) & \cos\theta(s)
\end{pmatrix}\begin{pmatrix}
    m(s) \\ 
    n(s)
\end{pmatrix}. \]
Let $(\gamma_t,v_1,v_2)$ be a pseudo-spherical timelike framed immersion. 
Now we introduce another moving frame along $\gamma_t$ that will be the main tool in our results. Define $(f_1,f_2)\in \Delta_5$ by 
\begin{equation}\label{fsfortimelike}
    f_1(s)=\dfrac{n(s)v_1(s)-m(s)v_2(s)}{\sqrt{n^2(s)+m^2(s)}},\quad f_2(s)=\dfrac{m(s)v_1(s)+n(s)v_2(s)}{\sqrt{n^2(s)+m^2(s)}}
\end{equation}
where $(m(s),n(s))\neq (0,0)$ for all $s\in I$. Then $(\gamma_t,f_1,f_2):I\to AdS^3\times \Delta_5$ is a pseudo-spherical timelike framed immersion and $\{\gamma_t,f_1,f_2,\mu=\gamma_t\times f_1\times f_2\}$ is a pseudo-orthonormal frame along $\gamma_t$. The derivative formulas for this frame are given by
\begin{equation}\label{TimelikeNSF}
	\begin{pmatrix}
		\gamma_t'(s)\\
	f_1'(s)\\
		f_2'(s)\\
        \mu'(s)
	\end{pmatrix}= \begin{pmatrix}
		0 & 0 & 0 & \alpha(s) 
        \\ 0 & 0 & \hat{\ell}(s) & 0
        \\ 0 & -\hat{\ell}(s) & 0 & \hat{n}(s)
        \\ -\alpha(s) &  0 &   \hat{n}(s) & 0
        
	\end{pmatrix}\begin{pmatrix}
	\gamma_t(s)\\
		f_1(s)\\
		f_2(s)\\
        \mu(s)
	\end{pmatrix},
\end{equation}
where 
\[ \hat{\ell}(s)=\ell(s)+\dfrac{m(s)n'(s)-n(s)m'(s)}{n^2(s)+m^2(s)},\quad \hat{n}(s)=\sqrt{n^2(s)+m^2(s)}.  \]
We finally define parallel curves of pseudo-spherical timelike framed curves. Let $(\gamma_t,v_1,v_2):I\to AdS^3\times \Delta_5$ be a pseudo-spherical timelike framed curve with the curvature $(\alpha,\ell,m,n)$. Consider the mapping $\gamma_t^\phi:I\to AdS^3$ defined by
\begin{equation*}\label{timeparallel}
    \gamma_t^\phi(s)=\cosh\phi\gamma_t(s)+\sinh\phi\left( \cos\theta(s) v_1(s)+\sin\theta(s)v_2(s) \right),
\end{equation*}
where $\phi$ is a fixed real number and $\theta'(s)=-\ell(s)$.  This mapping is called the anti-de Sitter parallel of $\gamma_t$. 
\begin{proposition}
    Let $(\gamma_t,v_1,v_2):I\to AdS^3\times \Delta_5$ be a pseudo-spherical timelike framed curve with the curvature $(\alpha,\ell,m,n)$. For a fixed real number $\phi$ and $\theta'(s)=-\ell(s)$, the anti-de Sitter parallel $(\gamma_t^\phi,v_1^\phi,v_2^\phi):I\to AdS^3\times \Delta_5$ is a pseudo-spherical timelike framed curve with the curvature $(\alpha^\phi,0,m^\phi, n^\phi)$, where
    \begin{align*}
        v_1^\phi(s)&= \sinh\phi\gamma_s(s)+\cosh\phi\left(\cos\theta(s) v_1(s)+\sin\theta(s)v_2(s)  \right),\\
        v_2^\phi(s)&=-\sin\theta(s) v_1(s)+\cos\theta(s)v_2(s),\\
        \alpha^\phi(s)&=\cosh\phi\,\alpha(s)+\sinh\phi(\cos\theta(s) m(s)+\sin\theta(s)n(s)),\\
        m^\phi(s)&= \sinh\phi\alpha(s)+\cosh\phi(\cos\theta(s) m(s)+\sin\theta(s)n(s)),\\
        n^\phi(s)&= -\sin\theta(s) m(s)+\cos\theta(s)n(s).
    \end{align*}
\end{proposition}
\begin{proof}
The proof of this proposition is quite similar to the proof of Proposition \ref{spaceparallelprop}.
\end{proof}
\section{Evolutes and focal surfaces of pseudo-spherical spacelike framed immersions in the anti-de Sitter 3-space}
We now define evolutes of pseudo-spherical spacelike framed immersions in $AdS^3$ and investigate properties of these evolutes. Throughout this section we assume $\alpha^2(s)+\hat{\epsilon}\hat{n}^2(s)\neq 0$ for all $s\in I$ unless otherwise stated.
\begin{definition}
    The \textit{total evolute} $\mathcal{E}(\gamma_s)$ of a pseudo-spherical spacelike framed immersion $(\gamma_s, f_1,f_2)$ is defined by
   \begin{equation}\label{evolutespace}
        \mathcal{E}(\gamma_s)(s)=\pm \dfrac{\hat{\ell}(s)\hat{n}(s)\,(\hat{n}(s)\gamma_s(s)-\alpha(s)f_2(s))-(\alpha(s)\hat{n}'(s)-\hat{n}(s)\alpha'(s))f_1(s)}{\sqrt{\big|\hat{\epsilon}(\alpha(s)\hat{n}'(s)-\hat{n}(s)\alpha'(s))^2-\hat{\ell}^2(s)\hat{n}^2(s)(\hat{n}^2(s)+\hat{\epsilon}\alpha^2(s))\big|}},
    \end{equation}
    where $g(s):=\hat{\epsilon}(\alpha(s)\hat{n}'(s)-\hat{n}(s)\alpha'(s))^2-\hat{\ell}^2(s)\hat{n}^2(s)(\hat{n}^2(s)+\hat{\epsilon}\alpha^2(s))\neq 0$.
    If $g(s)>0$, then $\mathcal{E}(\gamma_s)(s)\in S^3_2$. In this case we denote it by $ \mathcal{E}_p(\gamma_s)(s)$ and call it the \textit{PS-evolute} of $\gamma_s$. If $g(s)<0$, then $\mathcal{E}(\gamma_s)(s)\in AdS^3$. In this case we denote it by $ \mathcal{E}_a(\gamma_s)(s)$ and we call it the \textit{AdS-evolute} of $\gamma_s$. Note that if $g(s)>0$ and $\hat{\epsilon}=-1$, then we must assume that $\hat{n}^2(s)<\alpha^2(s)$ since otherwise $\mathcal{E}(\gamma_s)$ is not well-defined. 
\end{definition}
\begin{proposition}\label{propsevolparam}
Let $(\gamma_s,f_1,f_2):I\to AdS^3\times \Delta_1$ be a pseudo-spherical spacelike framed immersion with the curvature $(\hat{\alpha},\hat{\ell},\hat{m},\hat{n})$. Then the totally evolute $\mathcal{E}(\gamma_s)$ of $\gamma_s$ is independent of the parametrization of $(\gamma_s, f_1, f_2)$.
\end{proposition}
\begin{proof}
Let $(\Tilde{\gamma}_s, \Tilde{f}_1,\Tilde{f}_2)$ be a pseudo-spherical spacelike framed immersion with the curvature $(\hat{\alpha}_1,\hat{\ell}_1,\hat{m}_1,\hat{n}_1)$. Suppose that $(\gamma_s,f_1,f_2)$ and $(\Tilde{\gamma}_s,\Tilde{f}_1,\Tilde{f}_2)$ are parametrically equivalent by a change of parameter $u:\Tilde{I}\to I$. Similar to Proposition \ref{propsparametrization}, we have 
\begin{equation}\label{evparam1}
    \hat{\alpha}_1(s)=u'(s)\hat{\alpha}(u(s)),\:\: \hat{\ell}_1(s)=u'(s) \hat{\ell}(u(s)),\:\: \hat{m}_1(s)=u'(s)\hat{m}(u(s)),\:\:  \hat{n}_1(s)=u'(s)\hat{n}(u(s)).
\end{equation}
Moreover we find that
\begin{align*}
     \hat{\alpha}_1'(s)&=u''(s)\hat{\alpha}(u(s))+(u'(s))^2 \hat{\alpha}'(u(s)), \\
     \hat{n}_1'(s)&=u''(s)\hat{n}(u(s))+(u'(s))^2 \hat{n}'(u(s)),
\end{align*}
which from \eqref{evparam1} yields
\begin{equation}\label{evparam2}
  \hat{\alpha}_1(s)\hat{n}_1'(s)-\hat{n}_1(s)\hat{\alpha}_1'(s)=(u'(s))^3 \left(\hat{\alpha}'(u(s))\hat{n}(u(s))-\hat{\alpha}(u(s))\hat{n}'(u(s))\right).
\end{equation}
It follows easily from \eqref{evparam1} and \eqref{evparam2} that
\[ \mathcal{E}(\tilde{\gamma}_s)(s)=\mathcal{E}(\gamma_s)(u(s)). \]
Therefore the totally evolute of a pseudo-spherical spacelike framed immersion is independent of the parametrization.
\end{proof}
\begin{theorem}\label{evolatimeframed}
The AdS-evolute $\mathcal{E}_a(\gamma_s)$ of $\gamma_s$ is a pseudo-spherical framed base curve in $AdS^3$. More precisely, $(\mathcal{E}_a(\gamma_s),\mu,\eta)$ is a pseudo-spherical (spacelike or timelike) framed immersion with the curvature $(\alpha_{\mathcal{E}_a},\hat{\ell}_{\mathcal{E}_a},0,\hat{n}_{\mathcal{E}_a})$, where 
\begin{align*}
    \eta(s)&=\dfrac{\alpha(s) \gamma_s(s)+\hat{\epsilon}\hat{n}(s)f_2(s)}{\sqrt{|\alpha^2(s)+\hat{\epsilon}\hat{n}^2(s)|}}, \\
    \mu_{\mathcal{E}_a}(s)&=\dfrac{(\alpha(s)\hat{n}'(s)-\hat{n}(s)\alpha'(s))(\hat{n}(s)\gamma_s(s)-\alpha(s)f_2(s))-\hat{\ell}(s)\hat{n}(s)(\hat{\epsilon}\hat{n}^2(s)+\alpha^2(s))f_1(s)}{\sqrt{|\alpha^2(s)+\hat{\epsilon}\hat{n}^2(s)|}\sqrt{\hat{\ell}^2(s)\hat{n}^2(s)(\hat{n}^2(s)+\hat{\epsilon}\alpha^2(s))-\hat{\epsilon}(\alpha(s)\hat{n}'(s)-\hat{n}(s)\alpha'(s))^2}},\\
    \alpha_{\mathcal{E}_a}(s)&=\hat{\epsilon}\sqrt{|\alpha^2+\hat{\epsilon}n^2|}\dfrac{\hat{n}\big( 2\hat{\ell}\alpha'\hat{n}'+\hat{n}(\alpha'\hat{\ell}'-\hat{\ell}\alpha'') \big)-\alpha\big(-\hat{\ell}^3\hat{n}^2+\hat{n}\hat{\ell}'\hat{n}'+\hat{\ell}(2(\hat{n}')^2-\hat{n}\hat{n}'') \big)}{\hat{\ell}^2(s)\hat{n}^2(s)(\hat{n}^2(s)+\hat{\epsilon}\alpha^2(s))-\hat{\epsilon}(\alpha\hat{n}'(s)-\hat{n}(s)\alpha'(s))^2 }, \\
%    \alpha_{\mathcal{E}_a}(s)&=-\left(\dfrac{\alpha(s)\hat{\ell}(s)}{\sqrt{\hat{n}^2(s)-\alpha^2(s)}}+\dfrac{\omega'(s)}{1-\omega^2(s)}\right),\quad \omega(s)=\dfrac{-\alpha(s)\hat{n}'(s)+\hat{n}(s)\alpha'(s)}{\hat{\ell}(s)\hat{n}(s)\sqrt{\hat{n}^2(s)-\alpha^2(s)}},\\    
    \hat{\ell}_{\mathcal{E}_a}(s)&=\sqrt{|\alpha^2(s)+\hat{\epsilon}\hat{n}^2(s)|},\\
    \hat{n}_{\mathcal{E}_a}(s)&= -\hat{\epsilon}\dfrac{\sqrt{\hat{\ell}^2(s)\hat{n}^2(s)(\hat{n}^2(s)+\hat{\epsilon}\alpha^2(s))-\hat{\epsilon}(\alpha(s)\hat{n}'(s)-\hat{n}(s)\alpha'(s))^2}}{\alpha^2(s)+\hat{\epsilon}\hat{n}^2(s)}.
\end{align*}
\end{theorem}
\begin{proof}
   We directly have $\langle \mathcal{E}_a(\gamma_s),\mu \rangle=0$ and $\langle \mathcal{E}_a(\gamma_s), \eta \rangle=0$ since $\{\gamma_s, f_1,f_2,\mu \}$ is a pseudo-orthonormal frame. By differentiating $\mathcal{E}_a(\gamma_s)$ we find that
\[ \mathcal{E}_a'(\gamma_s)(s)=\Lambda(s)\left((\alpha(s)\hat{n}'(s)-\hat{n}(s)\alpha'(s))(\hat{n}(s)\gamma_s(s)-\alpha(s)f_2(s))-\hat{\ell}(s)\hat{n}(s)(\hat{\epsilon}\hat{n}^2(s)+\alpha^2(s))f_1(s)\right), \]
where
\[ \Lambda(s)=\hat{\epsilon}\dfrac{\hat{n}\big( 2\hat{\ell}\alpha'\hat{n}'+\hat{n}(\alpha'\hat{\ell}'-\hat{\ell}\alpha'') \big)-\alpha\big(-\hat{\ell}^3\hat{n}^2+\hat{n}\hat{\ell}'\hat{n}'+\hat{\ell}(2(\hat{n}')^2-\hat{n}\hat{n}'') \big)}{\left(\hat{\ell}^2(s)\hat{n}^2(s)(\hat{n}^2(s)+\hat{\epsilon}\alpha^2(s))-\hat{\epsilon}(\alpha(s)\hat{n}'(s)-\hat{n}(s)\alpha'(s))^2\right)^{3/2}}. \]
Using this equation, we see that $\langle \mathcal{E}'_a(\gamma_s),\mu \rangle=0$ and 
\[  \langle \mathcal{E}'_a(\gamma_s),\eta \rangle=\Lambda(s)\big((\alpha(s)\hat{n}'(s)-\hat{n}(s)\alpha'(s))(-\hat{n}(s)\alpha(s)+\alpha(s)\hat{n}(s))\big)=0.\]
Note that $\mu$ is a spacelike vector, but $\eta$ can be spacelike or timelike. Hence  $(\mathcal{E}_a(\gamma_s),\mu,\eta)$ is a pseudo-spherical (spacelike or timelike) framed curve. We can easily calculate $\mu_{\mathcal{E}_a}=\mathcal{E}_a(\gamma_s)\times \mu\times\eta$. Or it is easy to show that $\mu_{\mathcal{E}_a}$ defined in this theorem is a unit vector and it is pseudo-orthogonal to $\mathcal{E}_a(\gamma_s)$, $\mu$, and $\eta$. Then we immediately get $\alpha_{\mathcal{E}_a}$ from the equality $\mathcal{E}_a'(\gamma_s)=\alpha_{\mathcal{E}_a}\mu_{\mathcal{E}_a}$. Similarly $\hat{\ell}_{\mathcal{E}_a}$ and $\hat{n}_{\mathcal{E}_a}$ can be directly calculated by using derivative formulas of the pseudo-orthonormal frame $\{\mathcal{E}_a(\gamma_s),\mu,\eta,\mu_{\mathcal{E}_a} \}$ along $\mathcal{E}_a(\gamma_s)$
\end{proof}
\begin{remark}
    Notice that if $\hat{\epsilon}=1$, i.e., $f_1$ is a spacelike vector, then $\eta$ is a timelike vector. In this case $(\mathcal{E}_a(\gamma_s),\mu,\eta)$ is a pseudo-spherical spacelike framed immersion. 
\end{remark}
\begin{remark}\label{remspotherevol}
    The PS-evolute $\mathcal{E}_p(\gamma_s)$ of $\gamma_s$ is also a framed immersion in $S_2^3$. We do not prove this fact here since our focus in this paper is on the pseudo-spherical framed curves in anti-de Sitter space.
\end{remark}
  \begin{proposition} \label{spacepropevsing}
    \begin{enumerate}
        \item[{\normalfont (i)}] If $\gamma_s$ has singularity at $s_0$, then 
        \[ \mathcal{E}(\gamma_s)(s_0)=\pm \dfrac{\hat{n}(s_0)\ell(s_0) \gamma_s(s_0)+\alpha'(s_0)f_1(s_0)}{\sqrt{|\hat{\epsilon}(\alpha'(s_0))^2-\hat{\ell}^2(s_0)\hat{n}^2(s_0)|}}. \]
        %where $\hat{\ell}(s_0)\neq0$ and $\hat{n}(s_0)\neq 0$.
       In this case $\mathcal{E}(\gamma_s)$ has also singularity at $s_0$ if and only if 
        \[ 2\hat{\ell}(s_0)\alpha'(s_0)\hat{n}'(s_0)+n(s_0)(\alpha'(s_0)\hat{\ell}'(s_0)-\hat{\ell}(s_0)\alpha''(s_0))=0.\]
        \item[{\normalfont (ii)}] If $f_1$ has singularity at $s_0$, then 
$\mathcal{E}(\gamma_s)(s_0)=\pm f_1(s_0).$
        In this case $\mathcal{E}(\gamma_s)$ has also singularity at $s_0$ if and only if $\hat{n}(s_0)\hat{\ell}'(s_0)=0$.
        \item[{\normalfont (iii)}] If $f_2$ has singularity at $s_0$, then 
$\mathcal{E}(\gamma_s)(s_0)=\pm f_1(s_0).$
        In this case $\mathcal{E}(\gamma_t)$ has also singularity at $s_0$.
    \end{enumerate}
\end{proposition}
\begin{proof}
    By a direct calculation we obtain
\[ \mathcal{E}'(\gamma_s)(s)=\Lambda(s)\left((\alpha(s)\hat{n}'(s)-\hat{n}(s)\alpha'(s))(\hat{n}(s)\gamma_s(s)-\alpha(s)f_2(s))-\hat{\ell}(s)\hat{n}(s)(\hat{\epsilon}\hat{n}^2(s)+\alpha^2(s))f_1(s)\right), \]
where
\[ \Lambda(s)=-\epsilon_{\mathcal{E}}\hat{\epsilon}\dfrac{\hat{n}\big( 2\hat{\ell}\alpha'\hat{n}'+\hat{n}(\alpha'\hat{\ell}'-\hat{\ell}\alpha'') \big)-\alpha\big(-\hat{\ell}^3\hat{n}^2+\hat{n}\hat{\ell}'\hat{n}'+\hat{\ell}(2(\hat{n}')^2-\hat{n}\hat{n}'') \big)}{\left(\epsilon_{\mathcal{E}}\left(\hat{\ell}^2\hat{n}^2(\hat{n}^2+\hat{\epsilon}\alpha^2)-\hat{\epsilon}(\alpha\hat{n}'-\hat{n}\alpha')^2 \right)\right)^{3/2}}, \]
where $\epsilon_{\mathcal{E}}=\langle \mathcal{E}(\gamma_s),\mathcal{E}(\gamma_s)\rangle$.
If $\gamma_s$ has singularity at $s_0$, then $\alpha(s_0)=0$. Introducing this into $\mathcal{E}(\gamma_s)(s)$ and $\mathcal{E}'(\gamma_s)(s)$ gives (i). 

If $f_1$ has singularity at $s_0$, then $\hat{\ell}(s_0)=0$. Hence substituting $\hat{\ell}(s_0)=0$ into $\mathcal{E}(\gamma_s)(s)$ and $\mathcal{E}'(\gamma_s)(s)$ directly yields (ii). 

From \eqref{SpacelikeNSF}, $f_2$ has singularity at $s_0$ if and only if $\hat{\ell}(s_0)=\hat{n}(s_0)=0$. So (iii) is a direct consequence of (ii).
\end{proof}
The following proposition gives the relationship between the evolutes of a given pseudo-spherical spacelike framed immersion and its parallel. We will not give the proof since it follows from a messy but straightforward calculation.
\begin{proposition}
    Let $(\gamma_s,v_1,v_2)$ be a pseudo-spherical spacelike framed immersion, and for a fixed real number $\phi$, let $(\gamma_s^\phi,v_1^\phi,v_2^\phi)$ be the parallel of $\gamma_s$. Then $\mathcal{E}(\gamma_s^\phi)(s)=\mathcal{E}(\gamma_s)(s)$. 
\end{proposition}
\subsection{Focal surfaces of pseudo-spherical spacelike framed immersions}\label{secspacefocal}
The set of singular values of the focal surface of a curve gives the evolute of the same curve. This fact has been proved for many types of curves in different spaces \cite{hayashi, honda2, honda3,peinew}. Our aim in this section is to show that this important relationship between evolutes and focal surfaces also holds for pseudo-spherical spacelike framed immersions. We also give relationships between singularities of the evolute and of the focal surface.

Let $(\gamma_s, f_1,f_2)$ be a pseudo-spherical spacelike framed immersion with the curvature $(\alpha, \hat{l},0,\hat{n})$ in $AdS^3$.  Define 
\[ \zeta(s)=\dfrac{\hat{n}(s)\gamma_s(s)-\alpha(s)f_2(s)}{\sqrt{\epsilon_\zeta (-\hat{n}^2(s)-\hat{\epsilon}\alpha^2(s))} }, \]
where $\epsilon_\zeta=\langle \zeta,\zeta\rangle=\text{sgn}(-\hat{n}^2(s)-\hat{\epsilon}\alpha^2(s))$. Note that for $\eta$ defined in Theorem \ref{evolatimeframed}, $\langle \zeta,\eta\rangle=0$ and $\epsilon_\zeta \hat{\epsilon}=\epsilon_\eta=\langle \eta,\eta\rangle$. So we consider the following cases depending upon $\epsilon_\zeta$ and $\hat{\epsilon}$. 
\subsubsection*{Case 1.}
Let $\epsilon_\zeta\hat{\epsilon}=1$. Then $\zeta$ and $f_1$ are both timelike vectors. In this case we define the focal surface of $\gamma_s$ by
\begin{equation}\label{spacefscase1}
    \mathcal{F}_1(s,\theta)=\cos\theta\,\zeta(s)+\sin\theta f_1(s).
\end{equation}
Note that $\mathcal{F}_1(s,\theta)\in AdS^3$.  
Similar to the case with evolutes, it is easy to show that this focal surface of a pseudo-spherical spacelike immersion is independent of the choice of parametrization. Let us calculate  the partial derivatives of $\mathcal{F}_1(s,\theta)$. Differentiating \eqref{spacefscase1} with respect to $s$ and using \eqref{SpacelikeNSF}, we find that
\begin{align}
    \dfrac{\partial}{\partial s}\mathcal{F}_1(s,\theta)=&\dfrac{-\cos\theta}{\hat{n}^2(s)-\alpha^2(s)} \bigg( \dfrac{ \alpha(s)\hat{n}'(s)-\hat{n}(s)\alpha'(s)}{\sqrt{
    \hat{n}^2(s)-\alpha^2(s)}}(\alpha(s)\gamma(s)-\hat{n}(s)f_2(s)) \nonumber \\ & +\alpha(s)\hat{\ell}(s)\sqrt{\hat{n}^2(s)-\alpha^2(s)}f_1(s)  \bigg)  +\sin\theta\hat{\ell}(s) f_2(s). \label{spacefscase1-2}
\end{align}
Differentiating \eqref{spacefscase1} with respect to $\theta$ yields
\begin{equation}
    \label{spacefscase1-3} 
    \dfrac{\partial}{\partial\theta}\mathcal{F}_1(s,\theta)=-\sin\theta\,\zeta(s)+\cos\theta f_1(s).
\end{equation}
The focal surface of $\gamma_s$ defined by \eqref{spacefscase1} has singularity at $(s_0,\theta_0)$ if and only if $\mathcal{F}_1\times \frac{\partial}{\partial s}\mathcal{F}_1\times \frac{\partial}{\partial \theta}\mathcal{F}_1(s_0,\theta_0)=0$. Then from \eqref{spacefscase1}, \eqref{spacefscase1-2}, and \eqref{spacefscase1-3}, this triple vector product is equal to 
\[ -\dfrac{1}{\sqrt{\hat{n}^2(s)-\alpha^2(s)}}\left(\dfrac{\cos\theta(\alpha(s)\hat{n}'(s)-\hat{n}(s)\alpha'(s))}{\sqrt{\hat{n}^2(s)-\alpha^2(s)}}+\sin\theta \hat{\ell}(s)\hat{n}(s)\right)\mu(s), \]
which becomes zero at $(s_0,\theta_0)$ if and only if 
\begin{equation}\label{spacefscase1-4}
    \cos\theta_0(\alpha(s_0)\hat{n}'(s_0)-\hat{n}(s_0)\alpha'(s_0))+\sin\theta_0 \hat{\ell}(s_0)\hat{n}(s_0)\sqrt{\hat{n}^2(s_0)-\alpha^2(s_0)}=0.
\end{equation}
Substituting this into \eqref{spacefscase1} gives
\begin{align*}
    \mathcal{F}_1(s_0,\theta_0)&=\pm \dfrac{\hat{\ell}(s_0)\hat{n}(s_0)\,(\hat{n}(s_0)\gamma_s(s_0)-\alpha(s_0)f_2(s_0))-(\alpha(s_0)\hat{n}'(s_0)-\hat{n}(s_0)\alpha'(s_0))f_1(s_0)}{\sqrt{(\alpha(s_0)\hat{n}'(s_0)-\hat{n}(s_0)\alpha'(s_0))^2+\hat{\ell}^2(s_0)\hat{n}^2(s_0)(\hat{n}^2(s_0)-\alpha^2(s_0))}} \\
    &=\mathcal{E}_a(\gamma_s)(s_0).
\end{align*}
So the set of singular points of $\mathcal{F}_1(s,\theta)$ coincides with  the AdS-evolute $\mathcal{E}_a(\gamma_s)(s)$.
\begin{theorem}\label{spfocalsin}Suppose that the focal surface $\mathcal{F}_1(\gamma_s)(s,\theta)$ has singularity at $(s_0,\theta_0)$. Then 
    \begin{enumerate} 
        \item[{\normalfont (i)}] $\mathcal{F}_1(\gamma_s)(s,\theta)$ is locally diffeomorphic to the cuspidal edge at $(s_0,\theta_0)$ if and only if $\alpha_{\mathcal{E}_a}(s_0)\neq 0$, i.e., the AdS-evolute $\mathcal{E}_a(\gamma_s)(s)$ is regular at $s_0$.
        \item[{\normalfont (ii)}] $\mathcal{F}_1(\gamma_s)(s,\theta)$ is locally diffeomorphic to the swallowtail at $(s_0,\theta_0)$ if and only if $\alpha_{\mathcal{E}_a}(s_0)= 0$ and $\alpha'_{\mathcal{E}_a}(s_0)\neq 0$. 
    \end{enumerate}
\end{theorem}
\begin{proof}
Suppose that $\mathcal{F}_1(\gamma_s)(s,\theta)$ has singularity at $(s_0,\theta_0)$. Then from \eqref{spacefscase1-4} we have
\[ \tan\theta_0=\dfrac{-\alpha(s_0)\hat{n}'(s_0)+\hat{n}(s_0)\alpha'(s_0)}{\hat{\ell}(s_0)\hat{n}(s_0)\sqrt{\hat{n}^2(s_0)-\alpha^2(s_0)}}.  \]
The proof of this theorem is based on the well-known criteria for the cuspidal edge and the swallowtail (See \cite{Izumiya-Circular,kokubu} for details).  We consider the signed density function 
\begin{align*}
    \lambda(s,\theta)&=\det\left(\mathcal{F}_1, \dfrac{\partial}{\partial s}\mathcal{F}_1,  \dfrac{\partial}{\partial \theta}\mathcal{F}_1, \mu\right)\\
    &=-\dfrac{1}{\sqrt{\hat{n}^2(s)-\alpha^2(s)}}\left(\dfrac{\cos\theta(\alpha(s)\hat{n}'(s)-\hat{n}(s)\alpha'(s))}{\sqrt{\hat{n}^2(s)-\alpha^2(s)}}+\sin\theta \hat{\ell}(s)\hat{n}(s)\right).
\end{align*}
Set $\lambda^{-1}(0)=\mathcal{S}(\mathcal{F}_1(\gamma_s))$. We see that $\mathcal{S}(\mathcal{F}_1(\gamma_s))=\{(s,\theta(s))\}$, where $\theta(s)$ is a function satisfying $\lambda(s,\theta(s))=0$. Then we have 
\[ \dfrac{\partial}{\partial\theta}\lambda(s,\theta)=-\dfrac{1}{\sqrt{\hat{n}^2(s)-\alpha^2(s)}}\left(\dfrac{-\sin\theta(\alpha(s)\hat{n}'(s)-\hat{n}(s)\alpha'(s))}{\sqrt{\hat{n}^2(s)-\alpha^2(s)}}+\cos\theta \hat{\ell}(s)\hat{n}(s)\right)\neq 0, \]
since $(\alpha(s)\hat{n}'(s)-\hat{n}(s)\alpha'(s),\hat{\ell}(s)\hat{n}(s))\neq (0,0)$. Therefore any $p\in \mathcal{S}(\mathcal{F}_1(\gamma_s))$ is non-degenerate. Let $p$ be a non-degenerate singular point. Then there exists a regular curve $c:I\to I\times\mathbb{R}\subset \mathbb{R}^2$ such that $c(s_0)=p$ and $\text{image}(c)=\mathcal{S}(\mathcal{F}_1(\gamma_s))$ near $p$. Let $c(s)=(s,\theta(s))$. Consider the null vector field $\xi:I\to\mathbb{R}^2$ along $c(s)$ given by $\xi(s)=(1,-\alpha(s)\hat{\ell}(s)/\sqrt{\hat{n}^2(s)-\alpha^2(s)})$. Then from \eqref{spacefscase1-4}
\begin{align*}
\det(c'(s_0), \xi(s_0))&=-\dfrac{\alpha(s_0)\hat{\ell}(s_0)}{\sqrt{\hat{n}^2(s_0)-\alpha^2(s_0)}}+\theta'(s_0) \\
&=-\dfrac{\alpha(s_0)\hat{\ell}(s_0)}{\sqrt{\hat{n}^2(s_0)-\alpha^2(s_0)}}
    +\dfrac{d}{ds}\arctan\left(\dfrac{-\alpha(s)\hat{n}'(s)+\hat{n}(s)\alpha'(s)}{\hat{\ell}(s)\hat{n}(s)\sqrt{\hat{n}^2(s)-\alpha^2(s)}}\right)\bigg|_{s=s_0}\\
    &=\alpha_{\mathcal{E}_a}(s_0).
\end{align*}
Thus from \cite[Theorem 6.1(A)]{Izumiya-Circular}, $\mathcal{F}_1(\gamma_s)$ is locally diffeomorphic to the cuspidal edge at $(s_0,\theta_0)$ if and only if $\alpha_{\mathcal{E}_a}(s_0)\neq 0$. 

From \cite[Theorem 6.1(B)]{Izumiya-Circular},  $\mathcal{F}_1(\gamma_s)$ is locally diffeomorphic to the swallowtail at $(s_0,\theta_0)$ if and only if $\det(c'(s_0), \xi(s_0))=\alpha_{\mathcal{E}_a}(s_0)=0$ and   $(d/ds)\det(c'(s), \xi(s))\big|_{s=s_0}=\alpha'_{\mathcal{E}_a}(s_0)\neq 0$.
\end{proof}
\subsubsection*{Case 2.}
Let $\epsilon_\zeta\hat{\epsilon}=-1$, i.e., $\epsilon_\zeta=-\hat{\epsilon}$. Then one of $\zeta$ and $f_1$ is a timelike vector, and the other one is a spacelike vector. In this case we have four subcases depending on the causal characters of $\zeta$ and $f_1$ and the causal character of the focal surface. We can express these four cases with a single equation:
\begin{equation}\label{spacefscase2}
    \mathcal{F}_2(s,\theta)=P(\theta)\,\zeta(s)+R(\theta) f_1(s),
\end{equation}
where $(P(\theta), R(\theta))\in \{ (\cosh\theta,\sinh\theta), (\sinh\theta,\cosh\theta)\}$.
It is easy to see that $\mathcal{F}_2(s,\theta)\in AdS^3$ or $\mathcal{F}_2(s,\theta)\in S^3_2$ depending upon $\hat{\epsilon}$, $P(\theta)$, and $R(\theta)$. This focal surface is also independent of the choice of parametrization. Taking partial derivatives of \eqref{spacefscase2}, we obtain
\begin{align*}
    \dfrac{\partial}{\partial s}\mathcal{F}_2(s,\theta)=&\dfrac{P(\theta)}{\hat{\epsilon}(\hat{n}^2(s)+\hat{\epsilon}\alpha^2(s))} \bigg( \dfrac{\alpha(s)\hat{n}'(s)-\hat{n}(s)\alpha'(s)}{\sqrt{\hat{\epsilon}(\hat{n}^2(s)+\hat{\epsilon}\alpha^2(s))}}(\alpha(s)\gamma(s)+\hat{\epsilon}\hat{n}(s)f_2(s)) \nonumber \\ & -\alpha(s)\hat{\ell}(s)\sqrt{\hat{\epsilon}(\hat{n}^2(s)+\hat{\epsilon}\alpha^2(s))}f_1(s)  \bigg)  +R(\theta)\hat{\ell}(s) f_2(s), \\
     \dfrac{\partial}{\partial\theta}\mathcal{F}_2(s,\theta)=&R(\theta)\,\zeta(s)+P(\theta) f_1(s).  
\end{align*}
From these partial derivatives and \eqref{spacefscase2}
\[ \mathcal{F}_2\times \frac{\partial}{\partial s}\mathcal{F}_2\times \frac{\partial}{\partial \theta}\mathcal{F}_2= \dfrac{R^2(\theta)-P^2(\theta)}{\sqrt{\hat{\epsilon}(\hat{n}^2(s)+\hat{\epsilon}\alpha^2(s))}}\left(\dfrac{P(\theta)(\alpha(s)\hat{n}'(s)-\hat{n}(s)\alpha'(s))}{\sqrt{\hat{\epsilon}(\hat{n}^2(s)+\hat{\epsilon}\alpha^2(s))}}+R(\theta) \hat{\ell}(s)\hat{n}(s)\right)\mu(s). \]
Therefore, the focal surface $\mathcal{F}_2$ has singularity at $(s_0,\theta_0)$ if and only if
\begin{equation*}
    P(\theta_0)(\alpha(s_0)\hat{n}'(s_0)-\hat{n}(s_0)\alpha'(s_0))+R(\theta_0) \hat{\ell}(s_0)\hat{n}(s_0)\sqrt{\hat{\epsilon}(\hat{n}^2(s)+\hat{\epsilon}\alpha^2(s))}=0.
\end{equation*}
Introducing this equation into \eqref{spacefscase2} simply gives the point $\mathcal{E}(\gamma_s)(s_0)$ on the evolute \eqref{evolutespace} of $\gamma_s$. We conclude again that the set of singular points of $\mathcal{F}_2(s,\theta)$ coincides with the evolute of $\gamma_s$. Following similar steps in Theorem \ref{spfocalsin}, one can easily prove a similar theorem for $\mathcal{F}_2(s,\theta)$.
\subsection{PS-height functions}
Let $(\gamma_s,f_1,f_2):I\to AdS^3\times \Delta_1$ be a pseudo-spherical spacelike framed immersion. We now show how to explain the focal surface and the evolute of $\gamma_s$ as a wavefront from the viewpoint of Legendrian singularity theory. Define two families of functions $F^T:I\times AdS^3\to\mathbb{R}$ by $F^T(s,\mathbf{v})=\langle \mu(s),\mathbf{v}\rangle$ called \textit{the PS-timelike height function}, and $F^S:I\times S^3_2\to\mathbb{R}$ by $F^S(s,\mathbf{v})=\langle \mu(s),\mathbf{v}\rangle$ called \textit{the PS-spacelike height function} on $(\gamma_s,f_1,f_2)$. By a direct calculation, we have the following proposition.
\begin{proposition}
   Let $(\gamma_s,f_1,f_2):I\to AdS^3\times \Delta_1$ be a pseudo-spherical spacelike framed immersion with $g(s):=\hat{\epsilon}(\alpha(s)\hat{n}'(s)-\hat{n}(s)\alpha'(s))^2-\hat{\ell}^2(s)\hat{n}^2(s)(\hat{n}^2(s)+\hat{\epsilon}\alpha^2(s))\neq 0$.
   \begin{enumerate}
       \item[{\normalfont (i)}] Suppose that $g(s_0)<0$ for $s_0\in I$.
       \begin{enumerate}
           \item[{\normalfont (a)}] $F^T(s_0,\mathbf{v}_0)=0$ if and only if there exist real numbers $a,b,$ and $c$ such that $\mathbf{v}_0=a\gamma_s(s_0)+bf_1(s_0)+cf_2(s_0)$ with $-a^2+\hat{\epsilon}b^2-\hat{\epsilon}c^2=-1$.
           \item[{\normalfont (b)}] $F^T(s_0,\mathbf{v}_0)=(\partial F^T/\partial s)(s_0,\mathbf{v}_0)=0$ if and only if there exist real numbers $a$ and $b$ such that $\mathbf{v}_0=a\gamma_s(s_0)+bf_1(s_0)+(-a \alpha(s_0)/\hat{n}(s_0))f_2(s_0)$ with $\hat{n}^2(s_0)(1+\hat{\epsilon}b^2)=a^2(\hat{n}^2(s_0)+\hat{\epsilon}\alpha^2(s_0))$. 
           \item[{\normalfont (c)}] $F^T(s_0,\mathbf{v}_0)=(\partial F^T/\partial s)(s_0,\mathbf{v}_0)=(\partial^2 F^T/\partial s^2)(s_0,\mathbf{v}_0)=0$ if and only if $\mathbf{v}_0=\mathcal{E}_a(\gamma_s)(s_0)$.
       \end{enumerate}
       \item[{\normalfont (ii)}] Suppose that $g(s_0)>0$ for $s_0\in I$.
       \begin{enumerate}
           \item[{\normalfont (a)}] $F^S(s_0,\mathbf{v}_0)=0$ if and only if there exist real numbers $a,b,$ and $c$ such that $\mathbf{v}_0=a\gamma_s(s_0)+bf_1(s_0)+cf_2(s_0)$ with $-a^2+\hat{\epsilon}b^2-\hat{\epsilon}c^2=1$.
           \item[{\normalfont (b)}]$F^S(s_0,\mathbf{v}_0)=(\partial F^S/\partial s)(s_0,\mathbf{v}_0)=0$ if and only if there exist real numbers $a$ and $b$ such that $\mathbf{v}_0=a\gamma_s(s_0)+bf_1(s_0)+(-a \alpha(s_0)/\hat{n}(s_0))f_2(s_0)$ with $\hat{n}^2(s_0)(\hat{\epsilon}b^2-1)=a^2(\hat{n}^2(s_0)+\hat{\epsilon}\alpha^2(s_0))$. Note that if $\hat{\epsilon}=-1$, then $\hat{n}^2(s_0)<\alpha^2(s_0)$ must be satisfied.
           \item[{\normalfont (c)}] $F^S(s_0,\mathbf{v}_0)=(\partial F^S/\partial s)(s_0,\mathbf{v}_0)=(\partial^2 F^S/\partial s^2)(s_0,\mathbf{v}_0)=0$ if and only if $\mathbf{v}_0=\mathcal{E}_p(\gamma_s)(s_0)$.
       \end{enumerate}
   \end{enumerate}
\end{proposition}
For both height functions $F^T$ and $F^S$ defined above, the discriminant sets $\mathcal{D}_{F^T}$
and $\mathcal{D}_{F^S}$ coincide with the images of the focal surfaces in Section \ref{secspacefocal}. Moreover the secondary discriminant sets  $\mathcal{D}^2_{F^T}$ and $\mathcal{D}^2_{F^S}$ coincide with the images of the evolutes $\mathcal{E}_a(\gamma_s)$ and $\mathcal{E}_p(\gamma_s)$, respectively. Here recall that for the discriminant set and the secondary discriminant set of a smooth function $F:(\mathbb{R}\times\mathbb{R}^r,(s_0,\mathbf{v}_0))\to \mathbb{R}$ are respectively defined by
\begin{align*}
    \mathcal{D}_F&=\left\{ \mathbf{v}\in \mathbb{R}^r\:\bigg|\: F=\dfrac{\partial}{\partial s}F=0\:\text{at $(s,\mathbf{v})$ for some $s$}\right\}, \\
    \mathcal{D}^2_F&=\left\{ \mathbf{v}\in \mathbb{R}^r\:\bigg|\: F=\dfrac{\partial}{\partial s}F=\dfrac{\partial^2}{\partial s^2}F=0\:\text{at $(s,\mathbf{v})$ for some $s$}\right\}.
\end{align*}
\begin{example}
Take the smooth curve $\gamma_s:I\to AdS^3$ defined by
\[ \gamma_s(s)=\dfrac{1}{\sqrt{2}}\left(\sqrt{1+s^4}, \sqrt{1+s^6}, s^2, s^3\right). \]
The derivative of this curve with respect to $s$ is
\[ \gamma_s'(s)=\dfrac{1}{\sqrt{2}}\left( \dfrac{2s^3}{\sqrt{1+s^4}}, \dfrac{3 s^5}{\sqrt{1+s^6}},2s,3s^2 \right). \]
Therefore, the curve $\gamma_s$ is singular at $s=0$. Define $v_1:I\to S^3_2$ and $v_2:I\to AdS^3$ by
\begin{align*}
    v_1(s)&=\dfrac{1}{\sqrt{2(8+18s^2+s^6)}}(s^3\sqrt{1+s^4},s^3\sqrt{1+s^6},s^5+6s, s^6-4),\\
    v_2(s)&=\dfrac{1}{\sqrt{8+18s^2+s^6}\sqrt{4+9s^2+13s^6}}\big(-\sqrt{1+s^4}(4+9s^2-2s^6), \sqrt{1+s^6}(4+9s^2-2s^6),\\
    &\qquad\qquad\qquad\qquad\qquad\qquad\qquad\qquad 2s^2(-2+3s^2+s^6),3s^3(-2+3s^2+s^6)\big).
\end{align*}
It is easy to see that $\langle v_1, \gamma_s\rangle=0$, $\langle v_2, \gamma_s\rangle=0$, $\langle v_1, \gamma_s'\rangle=0$, and $\langle v_2, \gamma_s'\rangle=0$. Thus $(\gamma_s,v_1,v_2):I\to AdS^3\times \Delta_1$ is a pseudo-spherical spacelike framed curve in $AdS^3$. From the triple vector product $\gamma_s\times v_1\times v_2$, we find that
\[ \mu(s)=\dfrac{\sqrt{1+s^4}\sqrt{1+s^6}}{\sqrt{4+9s^2+13s^6}}\left(\dfrac{2s^2}{\sqrt{1+s^4}},\dfrac{3s^4}{\sqrt{1+s^6}},2,3s\right). \]
The curvature of this pseudo-spherical spacelike framed curve is given by $(\alpha,\ell,m,n)$, where
\begin{align*}
    \alpha(s)&=\dfrac{s\sqrt{4+9s^2+13s^6}}{\sqrt{2}\sqrt{1+s^4}\sqrt{1+s^6}},\\
    \ell(s)&=\dfrac{6\sqrt{2}s^2(2-3s^2-s^6)}{(8+18s^2+s^6)\sqrt{4+9s^2+13s^6}},\\
    m(s)&=\dfrac{12+16s^4+21s^6+25s^{10}}{\sqrt{2}\sqrt{1+s^4}\sqrt{1+s^6}\sqrt{8+18s^2+s^6}\sqrt{4+9s^2+13s^6}},\\
    n(s)&=\dfrac{s(-16+30s^2+81s^4+58s^6+102s^8+65s^{12})}{\sqrt{1+s^4}\sqrt{1+s^6}\sqrt{8+18s^2+s^6}(4+9s^2+13s^6)}.
\end{align*}
So $(\gamma_s,v_1,v_2)$ is actually a pseudo-spherical spacelike framed immersion in $AdS^3$. The hyperbolic Hopf map \eqref{hopfmap} allows us to visualize the projection of $\gamma_s$ on the hyperbolic space $H^2(1/2)$. It is easy to see that
\[ \mathbf{h}(\gamma_s)=\dfrac{1}{2}\left( s^2(\sqrt{1+s^4}+s \sqrt{1+s^6}),s^2(s\sqrt{1+s^4}-\sqrt{1+s^6}),1+s^4+s^6\right)\in H^2(1/2), \]
which is visualized in Figure \ref{fig1}.
\begin{figure}[H]
		\centering
		\includegraphics[width=0.6\textwidth]{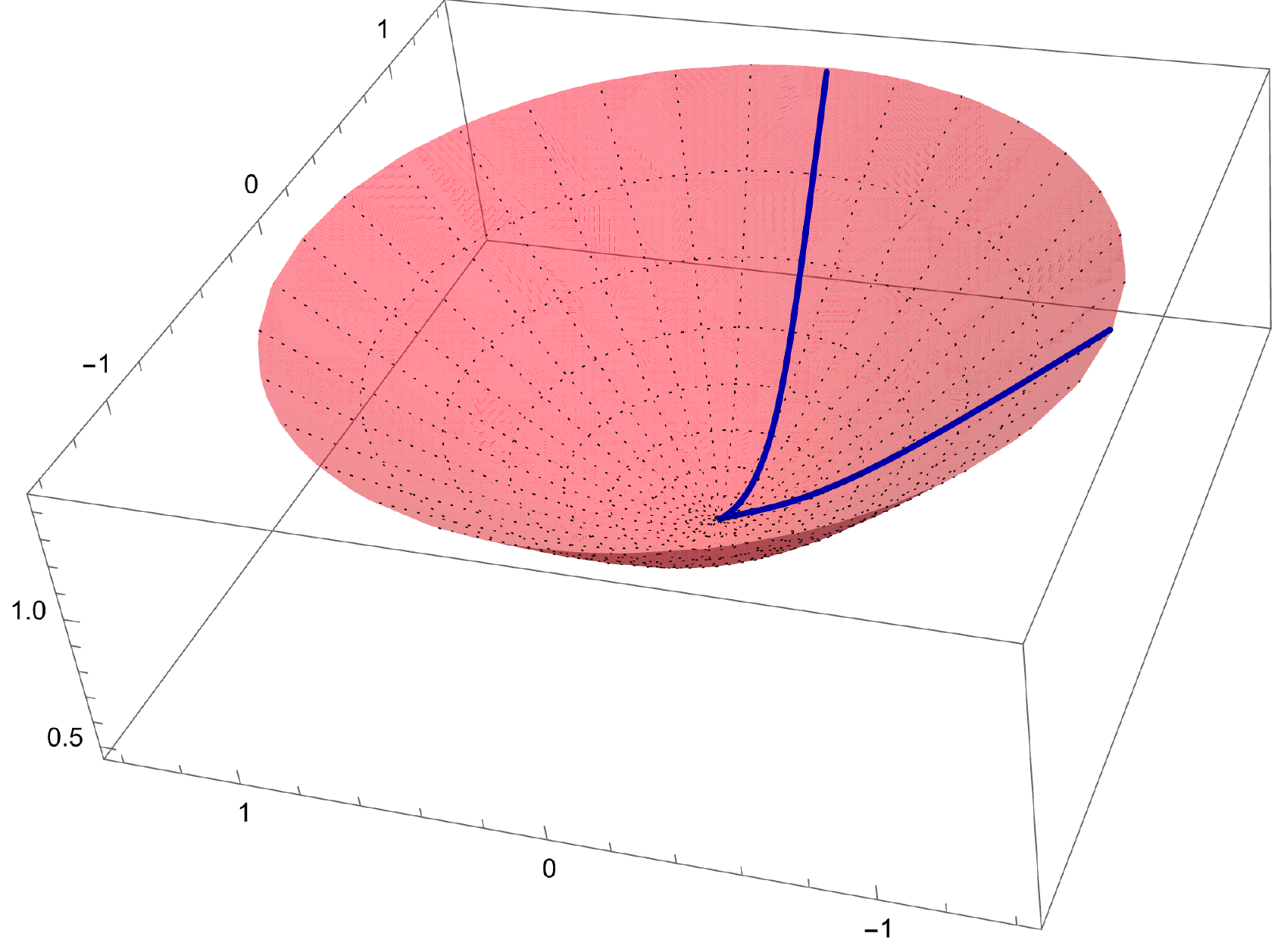}
		\caption{The projection of $\gamma_s$ on $H^2(1/2)$ by the hyperbolic Hopf map} \label{fig1}
	\end{figure}
The computation of the evolute of $\gamma_s$ is quite long and messy. We do this by using Wolfram-Mathematica and find the AdS-evolute of $\gamma_s$ which is very bulky to be written here. We just visualize of the projection on $H^2(1/2)$ of this evolute by using the hyperbolic hopf map (See Figure \ref{fig2}).
\begin{figure}[H]
		\centering
		\includegraphics[width=0.6\textwidth]{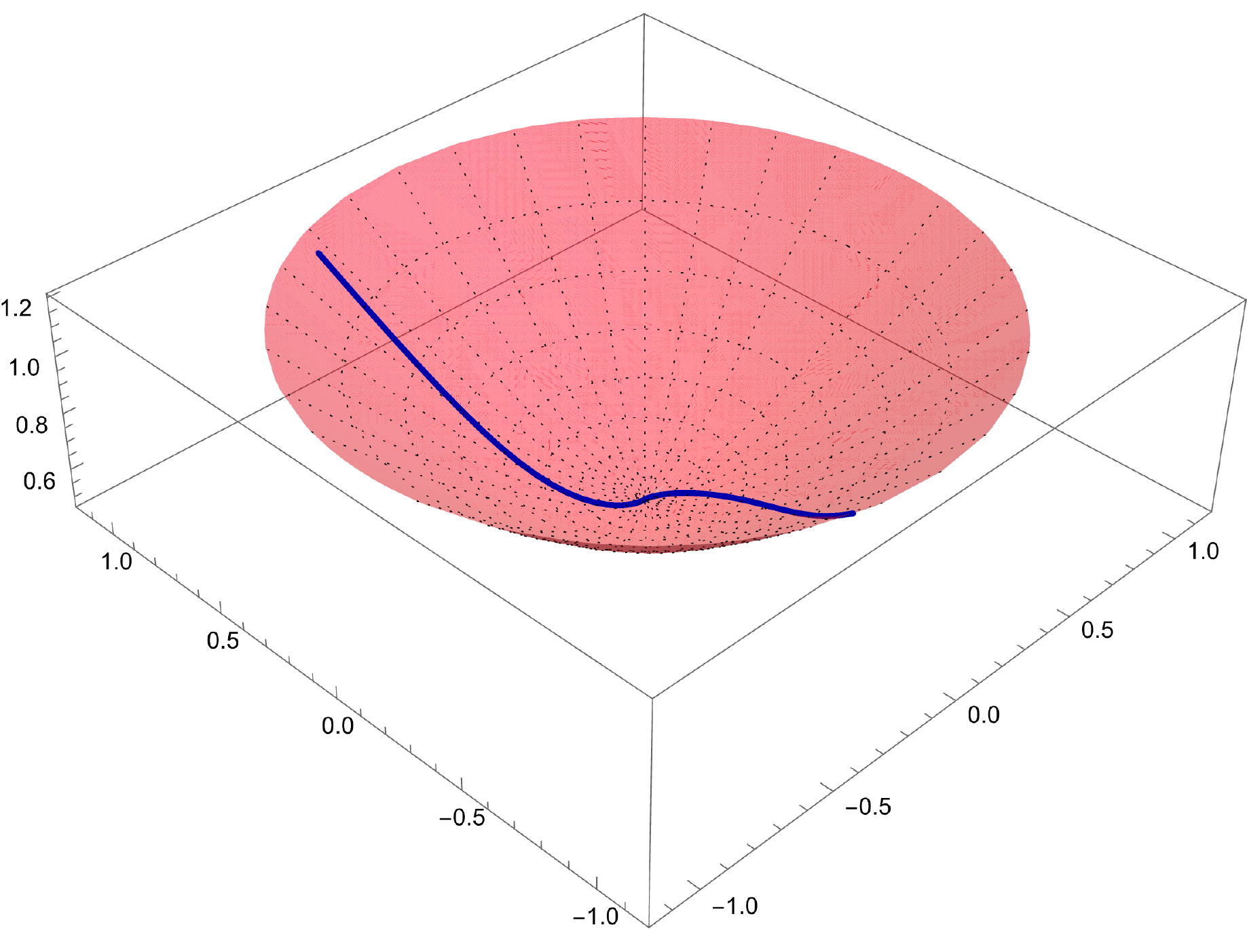}
		\caption{The projection of the anti-de Sitter evolute of $\gamma_s$ on $H^2(1/2)$ by the hyperbolic Hopf map} \label{fig2}
	\end{figure}
\end{example}
\section{Evolutes and focal surfaces of pseudo-spherical timelike framed immersions in the anti-de Sitter 3-space}

In this section, we introduce evolutes of pseudo-spherical timelike framed immersions in $AdS^3$ and investigate the properties of these evolutes. Throughout this section we assume $\alpha^2(s)\neq \hat{n}^2(s)$ for all $s\in I$ unless otherwise stated.

\begin{definition}
    The \textit{total evolute} $\mathcal{E}(\gamma_t)$ of a pseudo-spherical timelike framed immersion $(\gamma_t, f_1,f_2)$ is defined by
    \begin{equation}\label{evolutetime}
        \mathcal{E}(\gamma_t)(s)=\pm \dfrac{\hat{\ell}(s)\hat{n}(s)\,(\hat{n}(s)\gamma_t(s)-\alpha(s)f_2(s))-(\alpha(s)\hat{n}'(s)-\hat{n}(s)\alpha'(s))f_1(s)}{\sqrt{\big|(\alpha(s)\hat{n}'(s)-\hat{n}(s)\alpha'(s))^2-\hat{\ell}^2(s)\hat{n}^2(s)(\hat{n}^2(s)-\alpha^2(s))\big|}},
    \end{equation}
    where $f(s):=(\alpha(s)\hat{n}'(s)-\hat{n}(s)\alpha'(s))^2-\hat{\ell}^2(s)\hat{n}^2(s)(\hat{n}^2(s)-\alpha^2(s))\neq 0$.
    Notice that if $f(s)>0$, then $\mathcal{E}(\gamma_t)(s)\in S^3_2$. This evolute is denoted by $ \mathcal{E}_p(\gamma_t)(s)$ and called the \textit{PS-evolute} of $\gamma_t$. If $f(s)<0$, then $\mathcal{E}(\gamma_t)(s)\in AdS^3$. In this case, we denote this evolute by $ \mathcal{E}_a(\gamma_t)(s)$ and we call it the \textit{AdS-evolute} of $\gamma_t$. Note that if $f(s)<0$, then we must assume that $\hat{n}^2(s)>\alpha^2(s)$ since otherwise $\mathcal{E}(\gamma_t)$ is not well-defined.
\end{definition}
We will explain how Equation \eqref{evolutetime} can be derived shortly by using focal surfaces or certain height functions. For now, we investigate some geometric properties of these evolutes. The following proposition follows similarly to Proposition \ref{propsevolparam}.
\begin{proposition}
Let $(\gamma_t,f_1,f_2):I\to AdS^3\times \Delta_5$ be a pseudo-spherical timelike framed immersion with the curvature $(\hat{\alpha},\hat{\ell},\hat{m},\hat{n})$. Then the total evolute $\mathcal{E}(\gamma_t)$ of $\gamma_t$ is independent of the parametrization of $(\gamma_t, f_1, f_2)$.
\end{proposition}
\begin{theorem}\label{evolaframed}
The AdS-evolute $\mathcal{E}_a(\gamma_t)$ of $\gamma_t$ is a pseudo-spherical spacelike framed base curve in $AdS^3$. More precisely, $(\mathcal{E}_a(\gamma_t),\mu,\eta)$ is a pseudo-spherical spacelike framed immersion with the curvature $(\alpha_{\mathcal{E}_a},\hat{\ell}_{\mathcal{E}_a},0,\hat{n}_{\mathcal{E}_a})$, where 
\begin{align*}
    \eta(s)&=\dfrac{\alpha(s) \gamma_t(s)-\hat{n}(s)f_2(s)}{\sqrt{\hat{n}^2(s)-\alpha^2(s)}}, \\
    \mu_{\mathcal{E}_a}(s)&=\dfrac{(\alpha(s)\hat{n}'(s)-\hat{n}(s)\alpha'(s))(\hat{n}(s)\gamma_t(s)-\alpha(s)f_2(s))-\hat{\ell}(s)\hat{n}(s)(\hat{n}^2(s)-\alpha^2(s))f_1(s)}{\sqrt{\hat{n}^2(s)-\alpha^2(s)}\sqrt{\hat{\ell}^2(s)\hat{n}^2(s)(\hat{n}^2(s)-\alpha^2(s))-(\alpha(s)\hat{n}'(s)-\hat{n}(s)\alpha'(s))^2}},\\
    \alpha_{\mathcal{E}_a}(s)&=\sqrt{\hat{n}^2(s)-\alpha^2(s)}\dfrac{\hat{n}\big( 2\hat{\ell}\alpha'\hat{n}'-\hat{n}(-\alpha'\hat{\ell}'+\hat{\ell}\alpha'') \big)-\alpha\big(\hat{\ell}^3\hat{n}^2+\hat{n}\hat{\ell}'\hat{n}'+\hat{\ell}(2(\hat{n}')^2-\hat{n}\hat{n}'') \big)}{\hat{\ell}^2(s)\hat{n}^2(s)(\hat{n}^2(s)-\alpha^2(s))-(\alpha\hat{n}'(s)-\hat{n}(s)\alpha'(s))^2 }, \\
%    \alpha_{\mathcal{E}_a}(s)&=-\left(\dfrac{\alpha(s)\hat{\ell}(s)}{\sqrt{\hat{n}^2(s)-\alpha^2(s)}}+\dfrac{\omega'(s)}{1-\omega^2(s)}\right),\quad \omega(s)=\dfrac{-\alpha(s)\hat{n}'(s)+\hat{n}(s)\alpha'(s)}{\hat{\ell}(s)\hat{n}(s)\sqrt{\hat{n}^2(s)-\alpha^2(s)}},\\    
    \hat{\ell}_{\mathcal{E}_a}(s)&=-\sqrt{\hat{n}^2(s)-\alpha^2(s)},\\
    \hat{n}_{\mathcal{E}_a}(s)&= \dfrac{\sqrt{\hat{\ell}^2(s)\hat{n}^2(s)(\hat{n}^2(s)-\alpha^2(s))-(\alpha(s)\hat{n}'(s)-\hat{n}(s)\alpha'(s))^2}}{\alpha^2(s)-\hat{n}^2(s)}.
\end{align*}
\end{theorem}
\begin{proof}
We show that $(\mathcal{E}_a(\gamma_t),\mu,\eta)$ satisfies the conditions for being a pseudo-spherical spacelike framed immersion. It is easy to see that $\langle \mathcal{E}_a(\gamma_t),\mu \rangle=0$ and $\langle \mathcal{E}_a(\gamma_t), \eta \rangle=0$ since $\{\gamma_t, f_1,f_2,\mu \}$ is a pseudo-orthonormal frame. By a straightforward calculation we get
\[ \mathcal{E}_a'(\gamma_t)(s)=\Omega(s)\left((\alpha(s)\hat{n}'(s)-\hat{n}(s)\alpha'(s))(\hat{n}(s)\gamma(s)-\alpha(s) f_2(s))-\hat{\ell}(s)\hat{n}(s)(\hat{n}^2(s)-\alpha^2(s))f_1(s)\right), \]
where
\[ \Omega(s)=\dfrac{\hat{n}\big( 2\hat{\ell}\alpha'\hat{n}'-\hat{n}(-\alpha'\hat{\ell}'+\hat{\ell}\alpha'') \big)-\alpha\big(\hat{\ell}^3\hat{n}^2+\hat{n}\hat{\ell}'\hat{n}'+\hat{\ell}(2(\hat{n}')^2-\hat{n}\hat{n}'') \big)}{\left(\hat{\ell}^2(s)\hat{n}^2(s)(\hat{n}^2(s)-\alpha^2(s))-(\alpha\hat{n}'(s)-\hat{n}(s)\alpha'(s))^2 \right)^{3/2}}. \]
Hence we have $\langle \mathcal{E}'_a(\gamma_t),\mu \rangle=0$ and 
\[  \langle \mathcal{E}'_a(\gamma_t),\eta \rangle=\Omega(s)\big((\alpha(s)\hat{n}'(s)-\hat{n}(s)\alpha'(s))(-\hat{n}(s)\alpha(s)+\alpha(s)\hat{n}(s))\big)=0.\]
Therefore, since $\mu\in AdS^3$ and $\eta\in S^3_2$, $(\mathcal{E}_a(\gamma_t),\mu,\eta):I\to AdS^3\times \Delta_1$ is a pseudo-spherical spacelike framed curve. Then it is easy to calculate $\mu_{\mathcal{E}_a}=\mathcal{E}_a(\gamma_t)\times \mu\times\eta$ which directly gives $\alpha_{\mathcal{E}_a}$ from the equality $\mathcal{E}_a'(\gamma_t)=\alpha_{\mathcal{E}_a}\mu_{\mathcal{E}_a}$. Similarly $\hat{\ell}_{\mathcal{E}_a}$ and $\hat{n}_{\mathcal{E}_a}$ can be directly calculated by using derivative formulas of the pseudo-orthonormal frame $\{\mathcal{E}_a(\gamma_t),\mu,\eta,\mu_{\mathcal{E}_a} \}$ along $\mathcal{E}_a(\gamma_t)$.
\end{proof}
\begin{remark}
   Similar to the case in Remark \ref{remspotherevol}, one should expect that the PS-evolute $\mathcal{E}_p(\gamma_t)$ of $\gamma_t$ is also a framed immersion in $S_2^3$. We do not prove this fact here since our focus in this paper is on the pseudo-spherical framed curves in the anti-de Sitter 3-space.
\end{remark}
\begin{remark}\label{remalphaevol}
    Note that $\alpha_{\mathcal{E}_a}(s)$ in Theorem \ref{evolaframed} can be written a more compact form as
    \[\alpha_{\mathcal{E}_a}(s)=-\left(\dfrac{\alpha(s)\hat{\ell}(s)}{\sqrt{\hat{n}^2(s)-\alpha^2(s)}}+\dfrac{\omega'(s)}{1-\omega^2(s)}\right),\quad \omega(s)=\dfrac{-\alpha(s)\hat{n}'(s)+\hat{n}(s)\alpha'(s)}{\hat{\ell}(s)\hat{n}(s)\sqrt{\hat{n}^2(s)-\alpha^2(s)}}. \]
    Here for a point $s_0$ satisfying $\hat{\ell}(s_0)=\hat{n}(s_0)=0$, $\omega(s)$ is not well-defined since the denominator vanishes. However we exclude this case since simplifying $\omega'(s)/(1-\omega^2(s))$ will cancel out this vanishing term $\hat{\ell}(s)\hat{n}(s)$. Therefore we assume that $\alpha_{\mathcal{E}_a}(s)$ is well-defined at a point $s_0$ satisfying $\hat{\ell}(s_0)=\hat{n}(s_0)=0$. Our aim in writing $\alpha_{\mathcal{E}_a}(s)$ in a more compact form will become clear shortly.
\end{remark}

    \begin{proposition}
    \begin{enumerate}
        \item[{\normalfont (i)}] If $\gamma_t$ has singularity at $s_0$, then 
        \[ \mathcal{E}(\gamma_t)(s_0)=\pm \dfrac{\hat{n}(s_0)\ell(s_0) \gamma_t(s_0)+\alpha'(s_0)f_1(s_0)}{\sqrt{|(\alpha'(s_0))^2-\hat{\ell}^2(s_0)\hat{n}^2(s_0)|}}. \]
        %where $\hat{\ell}(s_0)\neq0$ and $\hat{n}(s_0)\neq 0$.
       In this case $\mathcal{E}(\gamma_t)$ has also singularity at $s_0$ if and only if 
        \[ 2\hat{\ell}(s_0)\alpha'(s_0)\hat{n}'(s_0)+n(s_0)(\alpha'(s_0)\hat{\ell}'(s_0)-\hat{\ell}(s_0)\alpha''(s_0))=0.\]
        \item[{\normalfont (ii)}] If $f_1$ has singularity at $s_0$, then 
$\mathcal{E}_p(\gamma_t)(s_0)=\pm f_1(s_0).$
        In this case $\mathcal{E}_p(\gamma_t)$ has also singularity at $s_0$ if and only if $\hat{n}(s_0)\hat{\ell}'(s_0)=0$.
        \item[{\normalfont (iii)}] If $f_2$ has singularity at $s_0$, then 
$\mathcal{E}_p(\gamma_t)(s_0)=\pm f_1(s_0).$
        In this case $\mathcal{E}_p(\gamma_t)$ has also singularity at $s_0$.
    \end{enumerate}
\end{proposition}
    \begin{proof}
    The proof of this proposition is very similar to the proof of Proposition \ref{spacepropevsing}. But here just notice that the AdS-evolute of $\gamma_t$ is not well-defined at a point $s_0$ such that $\hat{\ell}(s_0)=0$. 
    \end{proof}
%By a direct calculation we obtain
%\[ \mathcal{E}'(\gamma_t)(s)=\Omega(s)\left((\alpha(s)\hat{n}'(s)-\hat{n}(s)\alpha'(s))(\hat{n}(s)\gamma(s)-\alpha(s) f_2(s))-\hat{\ell}(s)\hat{n}(s)(\hat{n}^2(s)-\alpha^2(s))f_1(s)\right), \]
%where
%\[ \Omega(s)=\epsilon_{\mathcal{E}}\dfrac{\hat{n}\big( -2\hat{\ell}\alpha'\hat{n}'+\hat{n}(-\alpha'\hat{\ell}'+\hat{\ell}\alpha'') \big)+\alpha\big(\hat{\ell}^3\hat{n}^2+\hat{n}\hat{\ell}'\hat{n}'+\hat{\ell}(2(\hat{n}')^2-\hat{n}\hat{n}'') \big)}{\left(\epsilon_{\mathcal{E}}\left((\alpha\hat{n}'-\hat{n}\alpha')^2-\hat{\ell}^2\hat{n}^2(\hat{n}^2-\alpha^2) \right)\right)^{3/2}}. \]
%If $\gamma_t$ has singularity at $s_0$, then $\alpha(s_0)=0.$ Introducing this into $\mathcal{E}(\gamma_t)(s)$ and $\mathcal{E}'(\gamma_t)(s)$ gives (i). 

%If $f_1$ has singularity at $s_0$, then $\hat{\ell}(s_0)=0$. Notice that in this case the anti-de Sitter evolute is not well-defined at this point. Taking $\epsilon_{\mathcal{E}}=1$ and introducing $\hat{\ell}(s_0)=0$ into $\mathcal{E}_p(\gamma_t)(s)$ and $\mathcal{E}'_p(\gamma_t)(s)$ directly yields (ii). 

%From \eqref{TimelikeNSF}, $f_2$ has singularity at $s_0$ if and only if $\hat{\ell}(s_0)=\hat{n}(s_0)=0$. So (iii) is a direct consequence of (ii).
    
The following proposition follows from a messy but straightforward calculation.
\begin{proposition}
    Let $(\gamma_t,v_1,v_2)$ be a pseudo-spherical timelike framed immersion, and for a fixed real number $\phi$, let $(\gamma_t^\phi,v_1^\phi,v_2^\phi)$ be the parallel of $\gamma_t$. Then $\mathcal{E}(\gamma_t^\phi)(s)=\mathcal{E}(\gamma_t)(s)$. 
\end{proposition}
\subsection{Focal surfaces of pseudo-spherical timelike framed immersions}
In this section we obtain evolutes as the set of singular values of focal surfaces of pseudo-spherical timelike framed immersions. We also give relationships between singularities of the evolute and of the focal surface.

Let $(\gamma_t, f_1,f_2)$ be a pseudo-spherical timelike framed immersion with the curvature $(\alpha, \hat{l},0,\hat{n})$ in $AdS^3$.  Define 
\[ \zeta(s)=\dfrac{\hat{n}(s)\gamma_t(s)-\alpha(s)f_2(s)}{\sqrt{\epsilon_\zeta (\alpha^2(s)-\hat{n}^2(s))} }, \]
where $\epsilon_\zeta=\langle \zeta,\zeta\rangle=\text{sgn}(\alpha^2-\hat{n}^2)$. We will consider two cases depending upon $\epsilon_\zeta$. 
\subsubsection*{Case 1.}
Let $\epsilon_\zeta=-1$, i.e., $\zeta$ is a timelike vector. We define the focal surface $\mathcal{F}_3(s,\theta):I\times \mathbb{R}\to AdS^3$ of $\gamma_t$ as
\begin{equation}\label{timefscase1}
    \mathcal{F}_3(s,\theta)=\cosh\theta\,\zeta(s)+\sinh\theta f_1(s).
\end{equation}
Similar to the case with evolutes, it is easy to show that this focal surface of a pseudo-spherical timelike immersion is independent of the choice of parametrization. Now we find the partial derivatives of $\mathcal{F}_3(s,\theta)$. Differentiating \eqref{timefscase1} with respect to $s$ and using \eqref{TimelikeNSF} yields
\begin{align}
    \dfrac{\partial}{\partial s}\mathcal{F}_3(s,\theta)=&\dfrac{\cosh\theta}{\hat{n}^2(s)-\alpha^2(s)} \bigg( \dfrac{- (\alpha(s)\hat{n}'(s)-\hat{n}(s)\alpha'(s))}{\sqrt{\hat{n}^2(s)-\alpha^2(s)}}(\alpha(s)\gamma(s)-\hat{n}(s)f_2(s)) \nonumber \\ & +\alpha(s)\hat{\ell}(s)\sqrt{\hat{n}^2(s)-\alpha^2(s)}f_1(s)  \bigg)  +\sinh\theta\,\hat{\ell}(s) f_2(s). \label{timefscase1-2}
\end{align}
The derivative of \eqref{timefscase1} with respect to $\theta$ is easily obtained as
\begin{equation}
    \label{timefscase1-3} 
    \dfrac{\partial}{\partial\theta}\mathcal{F}_3(s,\theta)=\sinh\theta\zeta(s)+\cosh\theta f_1(s).
\end{equation}
The focal surface $\mathcal{F}_3$ of $\gamma_t$ has singularity at $(s_0,\theta_0)$ if and only if $\mathcal{F}_3\times \frac{\partial}{\partial s}\mathcal{F}_3\times \frac{\partial}{\partial \theta}\mathcal{F}_3(s_0,\theta_0)=0$. Then from \eqref{timefscase1}, \eqref{timefscase1-2}, and \eqref{timefscase1-3} 
\[ -\dfrac{1}{\sqrt{\hat{n}^2(s)-\alpha^2(s)}}\left(\dfrac{\cosh\theta(\alpha(s)\hat{n}'(s)-\hat{n}(s)\alpha'(s))}{\sqrt{\hat{n}^2(s)-\alpha^2(s)}}+\sinh\theta \hat{\ell}(s)\hat{n}(s)\right)\mu(s), \]
which vanishes at $(s_0,\theta_0)$ if and only if 
\begin{equation}\label{timefscase1-4}
    \cosh\theta_0(\alpha(s_0)\hat{n}'(s_0)-\hat{n}(s_0)\alpha'(s_0))+\sinh\theta_0 \hat{\ell}(s_0)\hat{n}(s_0)\sqrt{\hat{n}^2(s_0)-\alpha^2(s_0)}=0.
\end{equation}
Introducing this into \eqref{timefscase1}, we find that
\begin{align*}
    \mathcal{F}_3(s_0,\theta_0)&=\pm \dfrac{\hat{\ell}(s_0)\hat{n}(s_0)\,(\hat{n}(s_0)\gamma_t(s_0)-\alpha(s_0)f_2(s_0))-(\alpha(s_0)\hat{n}'(s_0)-\hat{n}(s_0)\alpha'(s_0))f_1(s_0)}{\sqrt{\hat{\ell}^2(s_0)\hat{n}^2(s_0)(\hat{n}^2(s_0)-\alpha^2(s_0))-(\alpha(s_0)\hat{n}'(s_0)-\hat{n}(s_0)\alpha'(s_0))^2}} \\
    &=\mathcal{E}_a(\gamma_t)(s_0).
\end{align*}
Therefore, the set of singular points of $\mathcal{F}_3(s,\theta)$ coincides with  the AdS-evolute $\mathcal{E}_a(\gamma_t)(s)$.
\begin{theorem}\label{timefocalsing}Suppose that the focal surface $\mathcal{F}_3(s,\theta)$ has singularity at $(s_0,\theta_0)$. Then
    \begin{enumerate}
        \item[{\normalfont (i)}] $\mathcal{F}_3(s,\theta)$ is locally diffeomorphic to the cuspidal edge at $(s_0,\theta_0)$ if and only if $\alpha_{\mathcal{E}_a}(s_0)\neq 0$, i.e., the AdS-evolute $\mathcal{E}_a(\gamma_t)(s)$ is regular at $s_0$.
        \item[{\normalfont (ii)}] $\mathcal{F}_3(s,\theta)$ is locally diffeomorphic to the swallowtail at $(s_0,\theta_0)$ if and only if $\alpha_{\mathcal{E}_a}(s_0)= 0$ and $\alpha'_{\mathcal{E}_a}(s_0)\neq 0$. 
    \end{enumerate}
\end{theorem}
\begin{proof}
If $\mathcal{F}_3(s,\theta)$ has singularity at $(s_0,\theta_0)$, then from \eqref{timefscase1-4} we have
\[ \tanh\theta_0=\dfrac{-\alpha(s_0)\hat{n}'(s_0)+\hat{n}(s_0)\alpha'(s_0)}{\hat{\ell}(s_0)\hat{n}(s_0)\sqrt{-\alpha^2(s_0)+\hat{n}^2(s_0)}}.  \]
We shall prove this theorem by using the well-known criteria for the cuspidal edge and the swallowtail (See \cite{Izumiya-Circular,kokubu} for details).  Now consider the signed density function 
\[ \lambda(s,\theta)=\det\left(\mathcal{F}_3, \dfrac{\partial}{\partial s}\mathcal{F}_3,  \dfrac{\partial}{\partial \theta}\mathcal{F}_3, \mu\right)=\dfrac{\cosh\theta(\alpha(s)\hat{n}'(s)-\hat{n}(s)\alpha'(s))}{\sqrt{-(\alpha^2(s)-\hat{n}^2(s))}}+\sinh\theta \hat{\ell}(s)\hat{n}(s). \]
Set $\lambda^{-1}(0)=\mathcal{S}(\mathcal{F}_3(\gamma_t))$. We see that $\mathcal{S}(\mathcal{F}_3(\gamma_t))=\{(s,\theta(s))\}$, where $\theta(s)$ is a function satisfying $\lambda(s,\theta(s))=0$. Then we have 
\[ \dfrac{\partial}{\partial\theta}\lambda(s,\theta)=\sinh\theta\dfrac{(\alpha(s)\hat{n}'(s)-\hat{n}(s)\alpha'(s))}{\sqrt{-(\alpha^2(s)-\hat{n}^2(s))}}+\cosh\theta \hat{\ell}(s)\hat{n}(s)\neq 0, \]
since $(\alpha(s)\hat{n}'(s)-\hat{n}(s)\alpha'(s),\hat{\ell}(s)\hat{n}(s))\neq (0,0)$. Therefore any $p\in \mathcal{S}(\mathcal{F}_3(\gamma_t))$ is non-degenerate. Let $p$ be a non-degenerate singular point. Then there exists a regular curve $c:I\to I\times\mathbb{R}\subset \mathbb{R}^2$ such that $c(s_0)=p$ and $\text{image}(c)=\mathcal{S}(\mathcal{F}_3(\gamma_t))$ near $p$. Let $c(s)=(s,\theta(s))$. Consider the null vector field $\xi:I\to\mathbb{R}^2$ along $c(s)$ given by $\xi(s)=(1,-\alpha(s)\hat{n}(s)/\sqrt{\hat{n}^2(s)-\alpha^2(s)})$. Then from \eqref{timefscase1-4} and Remark \ref{remalphaevol}
\begin{align*}
\det(c'(s_0), \xi(s_0))&=-\dfrac{\alpha(s_0)\hat{\ell}(s_0)}{\sqrt{\hat{n}^2(s_0)-\alpha^2(s_0)}}-\theta'(s_0) \\
&=-\dfrac{\alpha(s_0)\hat{\ell}(s_0)}{\sqrt{\hat{n}^2(s_0)-\alpha^2(s_0)}}
    -\dfrac{d}{ds}\arctanh\left(\dfrac{-\alpha(s)\hat{n}'(s)+\hat{n}(s)\alpha'(s)}{\hat{\ell}(s)\hat{n}(s)\sqrt{-\alpha^2(s)+\hat{n}^2(s)}}\right)\bigg|_{s=s_0}\\
    &=\alpha_{\mathcal{E}_a}(s_0).
\end{align*}
Thus from \cite[Theorem 6.1(A)]{Izumiya-Circular}, $\mathcal{F}_3(\gamma_t)$ is locally diffeomorphic to the cuspidal edge at $(s_0,\theta_0)$ if and only if $\alpha_{\mathcal{E}_a}(s_0)\neq 0$. 

From \cite[Theorem 6.1(B)]{Izumiya-Circular},  $\mathcal{F}_3(\gamma_t)$ is locally diffeomorphic to the swallowtail at $(s_0,\theta_0)$ if and only if $\det(c'(s_0), \xi(s_0))=\alpha_{\mathcal{E}_a}(s_0)=0$ and   $(d/ds)\det(c'(s), \xi(s))\big|_{s=s_0}=\alpha'_{\mathcal{E}_a}(s_0)\neq 0$.
\end{proof}
\begin{remark}
    For this case, it is also possible to construct another focal surface defined by
    \[ \mathcal{F}_4(s,\theta)=\sinh\theta\zeta(s)+\cosh\theta f_1(s). \]
    Notice that $\mathcal{F}_4(s,\theta)\in S^3_2$. The computations will be quite similar to those given above. A theorem similar to Theorem \ref{timefocalsing} can be easily obtained. 
\end{remark}
\subsubsection*{Case 2.}
Let $\epsilon_\zeta=1$, i.e., $\zeta$ is a spacelike vector. We define the focal surface $\mathcal{F}_5(s,\theta):I\times [0,2\pi)\to S^3_2$ of $\gamma_t$ as
\begin{equation*}\label{timefscase2}
    \mathcal{F}_5(s,\theta)=\cos\theta\,\zeta(s)+\sin\theta f_1(s).
\end{equation*}
To avoid repeating the same process, we leave analyzing this focal surface as an exercise to the reader. It is not surprising that the set of singular values of $\mathcal{F}_5$ coincides with the evolute $\mathcal{E}_p(\gamma_t)$, and a theorem similar to Theorem \ref{timefocalsing} is also satisfied for $\mathcal{F}_5$. 
\subsection{Anti-de Sitter height functions}
Let $(\gamma_t,f_1,f_2):I\to AdS^3\times \Delta_5$ be a pseudo-spherical timelike framed immersion. In this section we see that it is possible to explain the evolute of $\gamma_t$ as a wavefront from the viewpoint of Legendrian singularity theory as follows. We define two families of functions $H^T:I\times AdS^3\to\mathbb{R}$ by $H^T(s,\mathbf{v})=\langle \mu(s),\mathbf{v}\rangle$ called the \textit{AdS-timelike height function}, and $H^S:I\times S^3_2\to\mathbb{R}$ by $H^S(s,\mathbf{v})=\langle \mu(s),\mathbf{v}\rangle$ called the \textit{AdS-spacelike height function} on $(\gamma_t,f_1,f_2)$. The following proposition follows from a straightforward calculation.
\begin{proposition}
   Let $(\gamma_t,f_1,f_2):I\to AdS^3\times \Delta_5$ be a pseudo-spherical timelike framed immersion with $f(s)=(\alpha(s)\hat{n}'(s)-\hat{n}(s)\alpha'(s))^2-\hat{\ell}^2(s)\hat{n}^2(s)(\hat{n}^2(s)-\alpha^2(s))\neq 0$.
   \begin{enumerate}
       \item[{\normalfont (i)}] Suppose that $f(s)<0$ and $\hat{n}^2(s_0)>\alpha^2(s_0)$.
       \begin{enumerate}
           \item[{\normalfont (a)}] $H^T(s_0,\mathbf{v}_0)=0$ if and only if there exist real numbers $a,b,$ and $c$ such that $\mathbf{v}_0=a\gamma_t(s_0)+bf_1(s_0)+cf_2(s_0)$ with $-a^2+b^2+c^2=-1$.
           \item[{\normalfont (b)}] $H^T(s_0,\mathbf{v}_0)=(\partial H^T/\partial s)(s_0,\mathbf{v}_0)=0$ if and only if there exist real numbers $a$ and $b$ such that $\mathbf{v}_0=a\gamma_t(s_0)+bf_1(s_0)+(-a \alpha(s_0)/\hat{n}(s_0))f_2(s_0)$ with $\hat{n}^2(s_0)(1+b^2)=a^2(\hat{n}^2(s_0)-\alpha^2(s_0))$.
           \item[{\normalfont (c)}] $H^T(s_0,\mathbf{v}_0)=(\partial H^T/\partial s)(s_0,\mathbf{v}_0)=(\partial^2 H^T/\partial s^2)(s_0,\mathbf{v}_0)=0$ if and only if $\mathbf{v}_0=\mathcal{E}_a(\gamma_t)(s_0)$.
       \end{enumerate}
       \item[{\normalfont (ii)}] Suppose that $f(s)>0$.
       \begin{enumerate}
           \item[{\normalfont (a)}] $H^S(s_0,\mathbf{v}_0)=0$ if and only if there exist real numbers $a,b,$ and $c$ such that $\mathbf{v}_0=a\gamma_t(s_0)+bf_1(s_0)+cf_2(s_0)$ with $-a^2+b^2+c^2=1$.
           \item[{\normalfont (b)}] $H^S(s_0,\mathbf{v}_0)=(\partial H^S/\partial s)(s_0,\mathbf{v}_0)=0$ if and only if there exist real numbers $a$ and $b$ such that $\mathbf{v}_0=a\gamma_t(s_0)+bf_1(s_0)+(-a \alpha(s_0)/\hat{n}(s_0))f_2(s_0)$ with $\hat{n}^2(s_0)(b^2-1)=a^2(\hat{n}^2(s_0)-\alpha^2(s_0))$, where $b^2-1$ and $\hat{n}^2(s_0)-\alpha^2(s_0)$ have the same sign.
           \item[{\normalfont (c)}] $H^S(s_0,\mathbf{v}_0)=(\partial H^S/\partial s)(s_0,\mathbf{v}_0)=(\partial^2 H^S/\partial s^2)(s_0,\mathbf{v}_0)=0$ if and only if $\mathbf{v}_0=\mathcal{E}_p(\gamma_t)(s_0)$.
       \end{enumerate}
   \end{enumerate}
\end{proposition}
Notice that for both height functions $H^T$ and $H^S$ defined above, the discriminant sets $\mathcal{D}_{H^T}$ and $\mathcal{D}_{H^S}$ coincide with the images of the focal surfaces defined in the previous section. Moreover the secondary discriminant sets  $\mathcal{D}^2_{H^T}$ and $\mathcal{D}^2_{H^S}$ coincide with the images of the evolutes $\mathcal{E}_a(\gamma_t)$ and $\mathcal{E}_p(\gamma_t)$, respectively.
\begin{example}
    Consider the smooth curve $\gamma_t:I\to AdS^3$ defined by
\[ \gamma_t(s)=\left(\sqrt{2}\cosh(\frac{s}{\sqrt{2}}), \cosh(\sqrt{2}s)+\sqrt{2}\sinh{(\sqrt{2}s)}, \sqrt{2}\sinh{(\frac{s}{\sqrt{2}})}, \sqrt{2}\cosh{(\sqrt{2}s)}+\sinh(\sqrt{2}s)  \right). \]
The derivative of this curve with respect to $s$ is
\[ \gamma_t'(s)=\left( \sinh{(\frac{s}{\sqrt{2}})},2\cosh(\sqrt{2}s)+\sqrt{2}\sinh{(\sqrt{2}s)}, \cosh{(\frac{s}{\sqrt{2}})}, \sqrt{2}\cosh{(\sqrt{2}s)}+2\sinh(\sqrt{2}s) \right). \]
We see that $\langle \gamma_t',\gamma_t'\rangle=-1$, that is, the curve $\gamma_t$ is a regular timelike curve in $AdS^3$. We define $v_1:I\to S^3_2$ and $v_2:I\to S^3_2$ by
\begin{align*}
    v_1(s)&=\left(\cosh(\frac{s}{\sqrt{2}}), \sqrt{2}\cosh(\sqrt{2}s)+2\sinh{(\sqrt{2}s)}, \sinh(\frac{s}{\sqrt{2}}), \sqrt{2}\sinh(\sqrt{2}s)+2\cosh{(\sqrt{2}s)}\right),\\
    v_2(s)&=\bigg(-\sqrt{2}\sinh(\frac{s}{\sqrt{2}}), -(\sqrt{2}\cosh(\sqrt{2}s)+\sinh{(\sqrt{2}s)}), -\sqrt{2}\cosh(\frac{s}{\sqrt{2}}), \\
    &\qquad -(\sqrt{2}\sinh(\sqrt{2}s)+\cosh{(\sqrt{2}s)}) \bigg).
\end{align*}
So $\langle v_1, \gamma_s\rangle=0$, $\langle v_2, \gamma_s\rangle=0$, $\langle v_1, \gamma_s'\rangle=0$, and $\langle v_2, \gamma'\rangle=0$. Thus $(\gamma_t,v_1,v_2):I\to AdS^3\times \Delta_5$ is a pseudo-spherical timelike framed curve in $AdS^3$. We then find that
\[ \mu(s)=\left( \sinh{(\frac{s}{\sqrt{2}})},2\cosh(\sqrt{2}s)+\sqrt{2}\sinh{(\sqrt{2}s)}, \cosh{(\frac{s}{\sqrt{2}})}, \sqrt{2}\cosh{(\sqrt{2}s)}+2\sinh(\sqrt{2}s) \right). \]
The curvature of $\gamma_t$ is given by $(\alpha,\ell,m,n)$, where
\begin{equation*}
    \alpha(s)=1, \quad \ell(s)=1, \quad m(s)=3/\sqrt{2},\quad n(s)=0.
\end{equation*}
Thus $(\gamma_t,v_1,v_2)$ is a pseudo-spherical timelike framed immersion. 
Using \eqref{fsfortimelike}, we easily obtain
\[ f_1(s)=-v_2(s),\quad f_2(s)=v_1(s),\quad \hat{\ell}(s)=\ell(s),\quad \hat{n}(s)=m(s). \]
Therefore from \eqref{evolutetime} 
\begin{align*}
    \mathcal{E}_a(\gamma_t)(s)=&\bigg( \dfrac{2\sqrt{2}}{\sqrt{7}}\cosh(\frac{s}{\sqrt{2}}), \dfrac{\cosh(\sqrt{2}s)+\sqrt{2}\sinh(\sqrt{2}s)}{\sqrt{7}}, \dfrac{2\sqrt{2}}{\sqrt{7}}\sinh(\frac{s}{\sqrt{2}}), \\
    &\quad \dfrac{\sinh(\sqrt{2}s)+\sqrt{2}\cosh(\sqrt{2}s)}{\sqrt{7}} \bigg).
\end{align*}
By using the hyperbolic Hopf map \eqref{hopfmap}, we are able to visualize the projections of $\gamma_t$ and $\mathcal{E}_a(\gamma_t)$ on the hyperbolic space $H^2(1/2)$. See Figure \ref{fig3}.
\begin{figure}[H]
		\centering
		\includegraphics[width=0.6\textwidth]{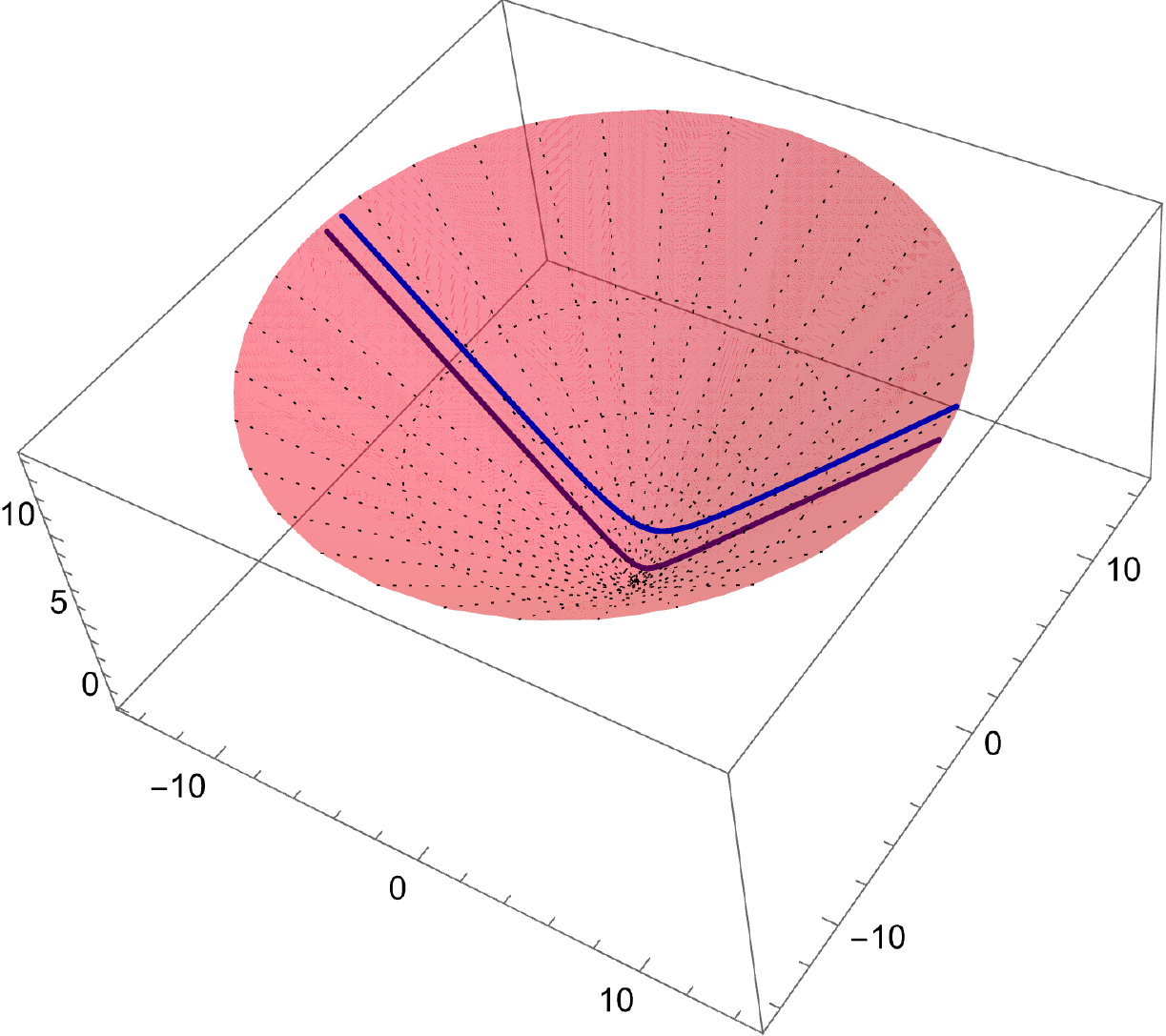}
		\caption{The projection of $\gamma_t$ (blue) and $\mathcal{E}_a(\gamma_t)$ (purple) on $H^2(1/2)$ by the hyperbolic Hopf map} \label{fig3}
	\end{figure}
\end{example}

\section*{Statements and Declarations}
\subsection*{Conflict of Interests} The author declares no known conflict of interest.
%\subsection*{Funding} There is no funding for this paper.

\end{document}